\documentclass[journal,twoside,web]{ieeecolor}
\usepackage{generic}
\usepackage{algorithmic}
\usepackage{textcomp}
\usepackage{cite}
\usepackage{xcolor}
\usepackage{fancyhdr}
\usepackage{subcaption}

\usepackage{amsmath,amsfonts,amsthm,amssymb,mathtools,bbm,dsfont,physics,aligned-overset}
\allowdisplaybreaks
\usepackage[T1]{fontenc}
\usepackage{color}
\usepackage[geometry]{ifsym}
\usepackage{siunitx}
\usepackage{multirow}
\usepackage{placeins}
\usepackage{booktabs}
\usepackage{rotating}
\usepackage{array}
\usepackage{enumerate}
\usepackage{dsfont}
\usepackage{bbm}
\usepackage{bbold}
\usepackage{caption,subcaption}
\usepackage[hyphens]{url}

\usepackage[capitalize]{cleveref}
\usepackage{comment}
\usepackage{microtype}

\usepackage{epsfig}
\usepackage{epstopdf}

\usepackage{xcolor}

\usepackage[bottom]{footmisc}
\usepackage{graphicx}

\DeclareMathOperator*{\argmin}{argmin}   
  
\usepackage{mathtools}
\usepackage[acronym]{glossaries}
\newtheorem{theorem}{Theorem}
\newtheorem{lemma}[theorem]{Lemma}
\newtheorem{corollary}[theorem]{Corollary}
\newtheorem{proposition}[theorem]{Proposition}

\theoremstyle{definition}
\newtheorem{definition}{Definition}
\newtheorem{assumption}{Assumption}
\newtheorem{example}{Example}
\newtheorem*{remark}{Remark}

\usepackage{stmaryrd}

\usepackage{tikz}
\usetikzlibrary{calc}
\usepackage{lipsum}

\newacronym{acr:dro}{DRO}{Distributionally Robust Optimization}
\newacronym{acr:drc}{DRC}{Distributionally Robust Control}
\newacronym{acr:dup}{DUP}{Distributional Uncertainty Propagation}
\newacronym{acr:kl}{KL}{Kullback-Leibler}
\newacronym{acr:tv}{TV}{Total Variation}
\newacronym{acr:lti}{LTI}{linear time-invariant}
\newacronym{acr:mmd}{MMD}{Maximum Mean Discrepancy}
\newacronym{acr:ols}{OLS}{ordinary least squares}
\newacronym{acr:wrls}{WRLS}{weighted and ridge-regularized least squares}
\newacronym{acr:ot}{OT}{Optimal Transport}
\newacronym{acr:cvar}{CVaR}{Conditional Value at Risk}

\newcommand{\reals}{\mathbb{R}}
\newcommand{\nonnegativeReals}{\mathbb{R}_{\geq 0}}

\newcommand{\naturals}{\mathbb{N}}
\newcommand{\Id}{\mathrm{Id}}
\renewcommand{\d}{\mathrm{d}}
\DeclareMathOperator*{\range}{Range}
\DeclareMathOperator*{\kernel}{Ker}
\newcommand{\ones}{\mathbbm{1}}

\newcommand{\eye}{I}
\newcommand{\innerProduct}[2]{\langle #1,#2\rangle}
\newcommand{\adjoint}[1]{#1^\ast}
\newcommand{\transpose}[1]{#1^\top}
\newcommand{\T}[1]{\transpose{#1}}
\newcommand{\inverse}[1]{#1^{-1}}
\newcommand{\inv}[1]{\inverse{#1}}
\newcommand{\pseudoinverse}[1]{#1^\dagger}
\newcommand{\pinv}[1]{\pseudoinverse{#1}}
\newcommand{\sample}[2]{\widehat{#1}^{(#2)}}
\newcommand{\entry}[2]{#1_{#2}}

\newcommand{\Q}{\mathbb{Q}}
\renewcommand{\P}{\mathbb{P}}
\newcommand{\expectedValue}[2]{\mathbb{E}_{#1}\left[#2\right]}

\newcommand{\diracDelta}[1]{\delta_{#1}}
\newcommand{\support}[1]{\text{supp}(#1)}

\newcommand{\tensorProd}[2]{#1\times #2}
\newcommand{\proj}[1]{\pi_{#1}}
\newcommand{\indicator}[1]{\ones_{#1}}

\newcommand{\setPlans}[2]{\Gamma(#1,#2)}
\newcommand{\probSpace}[1]{\mathcal{P}(#1)}
\newcommand{\Pp}[2]{\mathcal{P}_{#1}(#2)}
\newcommand{\pushforward}[1]{{#1}_{\#}}
\newcommand{\transportCost}[3]{W_{#1}(#2,#3)}

\newcommand{\ball}[3]{\mathbb{B}_{#1}^{#2}(#3)}

\let\argmin\relax

\DeclareMathOperator*{\argmin}{arg\,min}

\DeclareMathOperator*{\st}{s.t.}
\DeclareMathOperator{\cvar}{CVaR}
\newcommand{\DS}{\displaystyle}

\glsdisablehyper

\begin{document}
\title{Distributional Uncertainty Propagation via Optimal Transport}
\author{Liviu Aolaritei$^{* \mathsection}$, Nicolas Lanzetti$^*$, Hongruyu Chen, and Florian D\"orfler
\thanks{$^*$: Equal contribution; $^\mathsection$: Corresponding author.}
\thanks{This work was supported by the Swiss National Science Foundation under NCCR Automation, grant agreement 51NF40\_180545.}
\thanks{The authors are with the Automatic Control Laboratory, Department of Electrical Engineering and Information Technology at ETH Z\"urich, Switzerland {\tt \{aliviu,lnicolas,hongrchen,dorfler\}@ethz.ch}.}%
}

\maketitle

\begin{abstract}
This paper addresses the limitations of standard uncertainty models, e.g., robust (norm-bounded) and stochastic (one fixed distribution, e.g., Gaussian), and proposes to model uncertainty via Optimal Transport (OT) ambiguity sets. These constitute a very rich uncertainty model, which enjoys many desirable geometrical, statistical, and computational properties, and which: (1) naturally generalizes both robust and stochastic models, and (2) captures many additional real-world uncertainty phenomena (e.g., black swan events). Our contributions show that OT ambiguity sets are also analytically tractable: they propagate easily and intuitively through linear and nonlinear (possibly corrupted by noise) transformations, and the result of the propagation is again an OT ambiguity set or can be tightly upper bounded by an OT ambiguity set. In the context of dynamical systems, our results allow us to consider multiple sources of uncertainty (e.g., initial condition, additive noise, multiplicative noise) and to capture in closed-form, via an OT ambiguity set, the resulting uncertainty in the state at any future time. Our results are actionable, interpretable, and readily employable in a great variety of computationally tractable control and estimation formulations. To highlight this, we study three applications in trajectory planning, consensus algorithms, and least squares estimation. We conclude the paper with a list of exciting open problems enabled by our results.
\end{abstract}


\section{Introduction}
\label{sec:introduction}

% address this problem for the first time...how did the others do it?

Uncertainty modeling is a fundamental step in any decision-making problem across all areas of science and engineering. Designing a good uncertainty model is a very challenging problem on its own, due to the many aspects and trade-offs that need to be taken into consideration. From a practical standpoint, the uncertainty model should be \emph{expressive} enough to capture relevant uncertainty behaviors present in real-world systems, while avoiding being too `coarse' and leading to overly conservative decisions. From a theoretical standpoint, the expressivity should not compromise the \emph{analytical} (e.g., propagate in closed-form through dynamical systems) and \emph{computational} (i.e., lead to tractable decision-making) properties of the model. Additionally, it is highly desirable to have a unique uncertainty model, which enables a \emph{unified} study of decision-making problems.

To further complicate the modeling process, the modeler often has access only to partial statistical information about the uncertainty. For example, the \emph{support} of the uncertainty might be known (e.g., from first principles), but not the likelihood of its possible realizations. Alternatively, a limited number of \emph{samples} from the uncertainty might be available (e.g., in data-driven scenarios), but no information on the support. Finally, the probability distribution of the uncertainty might be available at present, but no information about future \emph{distribution shifts}. In all these cases, one is confronted with \emph{distributional uncertainty}, whereby not only is the system affected by uncertainty, but also the underlying probability distribution of the uncertainty is unknown and only partially observable. All these aspects cannot be addressed in a unifying manner by either the robust (i.e., norm-bounded uncertainty) or stochastic (i.e., one fixed distribution, e.g., an empirical or Gaussian distribution) models, and call for a \emph{distributionally robust} uncertainty description.

In this paper, we propose to model uncertainty via \emph{\gls{acr:ot} ambiguity sets}. OT ambiguity sets are balls of probability distributions constructed using an OT distance (e.g., the celebrated \emph{Wasserstein} distance), and centered at a reference probability distribution. These constitute a very rich uncertainty model, which generalizes both the robust and stochastic models. Moreover, they capture many additional real-world uncertainty phenomena, such as (specific types of) \emph{black swan events}, i.e., unpredictable rare events with dramatic consequences (see~\cref{subsec:properties:OT}). Finally, their inherent distributionally robust formulation naturally captures many classes of distribution shifts. All these properties highlight the expressivity of \gls{acr:ot} ambiguity sets.

So far, \gls{acr:ot} ambiguity sets have been primarily employed in~\emph{\gls{acr:dro}} problems, with many recent papers showing that this approach leads to computationally tractable formulations that successfully robustify state-of-the-art data-driven optimization and machine learning models; see \cite{Peyman2018,sinha2017certifying,Blanchet2019,gao2022finite,shafieezadeh2023new} and references therein. Recently, such methodology has also spread in control, establishing an exciting new direction in (OT-based) \emph{\gls{acr:drc}}. Existing work has been primarily focused on optimal control and MPC \cite{yang2020wasserstein,aolaritei2023wasserstein,zhong2023efficient,micheli2022data,fochesato2022data,mark2021data,coulson2021distributionally,navsalkar2023data,kim2023distributional,mcallister2023inherent,zolanvari2023iterative}, estimation and filtering \cite{shafieezadeh2018wasserstein,wang2022distributionally,han2023distributionally,lotidis2023wasserstein,brouillon2023regularization}, and uncertainty quantification \cite{aolaritei2023capture,aolaritei2022uncertainty,boskos2023high,boskos2020data}. For non-OT-based works, see \cite{VanParys2016,coppens2021data,schuurmans2023general,dixit2022distributionally,romao2023distributionally} and references therein.

Differently from \gls{acr:dro}, which is a static optimization problem, \gls{acr:drc} is a \emph{dynamic} problem involving the interplay between a stochastic dynamical system and a real-time distributionally robust (optimal) control algorithm. Due to its dynamic nature, a key challenge in developing and analysing \gls{acr:drc} methods is represented by the ability to predict the future behavior of the system. This, in turn, is accomplished by understanding how OT ambiguity sets propagate through the system dynamics.

In this paper, we address this problem in the context of both linear and nonlinear stochastic dynamical systems, with additive or multiplicative uncertainty.
%Moreover, we assume that only partial statistical information about the uncertainty is available (e.g., in the form of a limited number of samples).
To do so, we take a step back and start by studying how OT ambiguity sets propagate through (i) arbitrary deterministic maps $f(x)$ (which boil down to $f(x) = Ax$ in the linear case), and (ii) stochastic maps of the form $f(x) = x+y$ or $f(x) = x \cdot y$ (i.e., pointwise sum and product), for an uncertain random vector $y$. This is done in~\cref{sec:propagation,sec:additive:multiplicative}, respectively. We then specialize our results to stochastic dynamical systems and \gls{acr:drc} problems in Section~\ref{sec:applications}. The proposed framework is coined (OT-based) \emph{\gls{acr:dup}}, which we believe to be the fundamental link between static \gls{acr:dro} and dynamic \gls{acr:drc} problems (see \cref{fig:DUP}).

Our main contributions can be summarized as follows:

\begin{enumerate}
    \item[0)] \textbf{Expressivity of OT ambiguity sets}. We present a tutorial on the modeling power of OT ambiguity sets; see Section~\ref{sec:OT}.

    \smallskip

    \item[1)] \textbf{\gls{acr:dup} via deterministic nonlinear maps}. In~\cref{thm:nonlin:trans}, we study the propagation of \gls{acr:ot} ambiguity sets through both arbitrary and structured (bijective, injective, and surjective) nonlinear maps $f$. First, we show that \gls{acr:ot} ambiguity sets are \emph{exactly} (in closed-form) propagated through bijective maps, and they are \emph{closed} under propagation (i.e., the result of the propagation is again an \gls{acr:ot} ambiguity set). These properties are crucial features of \gls{acr:ot} ambiguity sets, which highlight their advantage over other stochastic uncertainty models, such as standard stochastic models consisting of a fixed distribution (e.g., Gaussian) or other types of distributionally robust descriptions (e.g., moment ambiguity sets). Secondly, we show that in the remaining cases, the result of the propagation can be upper bounded by an \gls{acr:ot} ambiguity set. 

    \smallskip

    \item[2)] \textbf{\gls{acr:dup} via deterministic linear maps}. In~\cref{thm:lin:trans}, we restrict our attention to linear maps and prove a chain of inclusions that strenghtens~\cref{thm:nonlin:trans} (when restricted to linear transformations). Moreover, we show that for linear maps the exactness and closedness of the propagation can be further extended from bijective to surjective maps.

    \smallskip

    \item[3)] \textbf{\gls{acr:dup} via stochastic additive and multiplicative maps}. In~\cref{thm:conv:trans,thm:Hadamard}, we study the convolution and Hadamard product of two \gls{acr:ot} ambiguity sets. These two probabilistic operations correspond to the pointwise sum and product of two uncertain random vectors, respectively. In both cases, we show that the result of the convolution can be upper-bounded by another \gls{acr:ot} ambiguity set. 

    \smallskip

    \item[4)] \textbf{Applications}. In~\cref{sec:applications}, we employ our \gls{acr:dup} results in the context of control and estimation problems. Specifically, we start by employing the results in 1)-3) for discrete-time \gls{acr:lti} systems with (i) uncertain initial condition, (ii) additive uncertainty, and (iii) multiplicative uncertainty. In all three cases, we show that the distributional uncertainty in the state can be captured in closed-form by an \gls{acr:ot} ambiguity set. We then study three distinct applications: trajectory planning, consensus, and least squares estimation. We conclude the paper with preliminary results for nonlinear systems and a discussion on future research directions.
\end{enumerate}

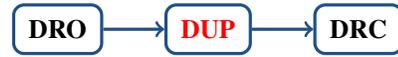
\begin{figure}[t]
    \centering
    \begin{tikzpicture}
        \definecolor{mycolor}{RGB}{30,81,136}
        \node[rectangle,rounded corners,draw=mycolor,very thick,inner sep=0.2cm] (a) at (-2,0) {\textbf{DRO}}; 
        \node[rectangle,rounded corners,draw=mycolor,very thick,inner sep=0.2cm] (b) at (0,0) {\color{red}{\textbf{DUP}}}; 
        \node[rectangle,rounded corners,draw=mycolor,very thick,inner sep=0.2cm] (c) at (2,0) {\textbf{DRC}}; 
        \draw[->,color=mycolor,very thick] (a) -- (b); 
        \draw[->,color=mycolor,very thick] (b) -- (c); 
    \end{tikzpicture}
    \caption{\gls{acr:dup} enables the use of static \gls{acr:dro} approaches in dynamic \gls{acr:drc} problems.}
    \label{fig:DUP}
\end{figure}

The \gls{acr:dup} results show that \gls{acr:ot} ambiguity sets are \emph{analytically tractable}, and constitute a promising well-posed uncertainty model for dynamical systems and control. Specifically, they allow to consider multiple sources of distributional uncertainty (e.g., from initial condition, additive noise, multiplicative noise) and capture via \gls{acr:ot} ambiguity sets the resulting distributional uncertainty in the state at any future time step. Moreover, our results are exact for a great variety of practical cases. 

For linear systems (and certain classes of nonlinear systems), the resulting \gls{acr:ot} ambiguity sets are also \emph{computationally tractable}, and can be directly employed in various \gls{acr:drc} formulations which can be reformulated as tractable convex programs using, for example, the \gls{acr:dro} results in \cite{shafieezadeh2023new} for many cases of practical interest. An initial effort in this direction has been made in \cite{aolaritei2023wasserstein}, where the \gls{acr:dup} framework has been exploited to formulate a Wasserstein Tube MPC which: (i) is a direct generalization of the deterministic Tube MPC to the stochastic setting, (ii) is recursively feasible, and (iii) can achieve an optimal trade-off between safety and performance.

Aside from the computational advantages, we envision that the \gls{acr:dup} framework can be directly employed in the analysis of the closed-loop properties (e.g., invariance, stability) of systems affected by distributional uncertainty.

\subsection{Mathematical Notation and Preliminaries}

Throughout the paper, whenever we use $\mathcal X$ and $\mathcal Y$ we implicitly assume that they are Polish spaces, i.e., separable and completely metrizable topological spaces (e.g., $\reals^n$ and $L^p$ spaces with $1\leq p<+\infty$). We denote by $\probSpace{\mathcal{X}}$ the space of Borel probability distributions on $\mathcal{X}$. The support of a distribution $\P\in\probSpace{\mathcal{X}}$ is denoted by $\support{\P}\subset \mathcal X$, and the delta distribution at $x\in\mathcal{X}$ is denoted by $\diracDelta{x}$. Given $\P, \Q \in \mathcal P(\mathcal X)$, we denote by $\P \otimes \Q$ their product distribution. %, and by $\P^{\otimes t}$ the $t$-fold product distribution $\P \otimes \ldots \otimes \P$ with $t$ terms. 
We denote by $x \sim \P$ the fact that $x$ is distributed according to $\P$. %Throughout the paper, we use the notation $\widehat{\P}$ to denote an empirical distribution constructed from samples. 
We implicitly assume that all maps $f:\mathcal X \to \mathcal Y$ are Borel. Projection maps are denoted by $\pi$, and the identity map on $\mathcal X$ is $\Id_\mathcal{X}$ (or simply $\Id$ whenever the context is clear). The indicator function of the set $\mathcal A$ is denoted by $\indicator{\mathcal A}$. Finally, given $x,y \in \reals^n$, $x \cdot y$ denotes their Hadamard (or pointwise) product, defined element-wise by $(x \cdot y)_i = x_i y_i$ for $i\in [n]:=\{1,\ldots,n\}$.

We focus on three classes of transformations of probability distributions: \emph{pushforward} via a transformation, \emph{convolution} with another distribution, and \emph{Hadamard product} with another distribution. We start by defining the pushforward.

\begin{definition}
\label{def:pushforwad}
Let $\P\in\probSpace{\mathcal{X}}$ and $f:\mathcal{X}\to\mathcal{Y}$. Then, the pushforward of $\P$ via $f$ is denoted by $\pushforward{f}\P$, and is defined as $(\pushforward{f}\P)(\mathcal A)\coloneqq \P(f^{-1}(\mathcal A))$, for all Borel sets $\mathcal A\subset\mathcal{Y}$.
\end{definition}

\cref{def:pushforwad} says that if $x\sim\P$, then $\pushforward{f}\P$ is the probability distribution of the random variable $y=f(x)$. %We exemplify this definition in the following example. 
We now consider $\mathcal X=\reals^n$ and define the convolution and the Hadamard product.

\begin{definition}
\label{def:convolution}
Let $\P, \Q\in\probSpace{\reals^n}$. The convolution $\P\ast\Q\in\probSpace{\reals^n}$ is defined by
$
    (\P\ast\Q)(\mathcal A)
    \coloneqq 
    \int_{\reals^n\times\reals^n}\indicator{\mathcal A}(x+y)\d(\P\otimes\Q)(x,y), 
$
for all Borel sets $\mathcal A\subset\reals^n$.
\end{definition}

\begin{definition}
\label{def:Hadamard}
Let $\P, \Q\in\probSpace{\reals^n}$. The Hadamard product $\P\odot\Q\in\probSpace{\reals^n}$ is defined by
$
    (\P\odot\Q)(\mathcal A)
    \coloneqq 
    \int_{\reals^n\times\reals^n}\indicator{\mathcal A}(x\cdot y)\d(\P\otimes\Q)(x,y), 
$
for all Borel sets $\mathcal A\subset\reals^n$.
\end{definition}

The convolution captures the sum of independent random variables: if $x \sim \P$ and $y \sim \Q$ are independent, then $x+y$ is distributed according to $\P\ast\Q$.
Similarly, the Hadamard product captures element-wise multiplication of independent random variable: if $x \sim \P$ and $y \sim \Q$ are independent, then $x \cdot y$ is distributed according to $\P\odot\Q$.
%Moreover, an example similar to \cref{ex:convolution:empirical} can be written for the Hadamard product, with $\ast$ and $+$ replaced by $\odot$ and $\cdot$, respectively.
We exemplify these three definitions in the case of empirical probability measures: 

\begin{example}
\label{ex:pushforwad:empirical}
Let $\P=\frac{1}{N}\sum_{i=1}^N\diracDelta{\sample{x}{i}}$ and $\Q=\frac{1}{M}\sum_{j=1}^M\diracDelta{\sample{y}{j}}$ be empirical distributions with $N$ and $M$ samples.
Then, $\pushforward{f}\P{}=\frac{1}{N}\sum_{i=1}^N\delta_{f(\sample{x}{i})}$, 
$\P\ast\Q = \frac{1}{NM}\sum_{i=1}^N\sum_{j=1}^M\diracDelta{\sample{x}{i}+\sample{y}{j}}$, and
$\P\odot\Q = \frac{1}{NM}\sum_{i=1}^N\sum_{j=1}^M\diracDelta{\sample{x}{i}\cdot\sample{y}{j}}$. In particular, all distributions are empirical and supported on the propagated samples.
\end{example}

% --------------------------------------------------------------------
% --------------------------------------------------------------------

\section{Optimal Transport Ambiguity Sets}
\label{sec:OT}

\begin{figure}[t]
    \centering
    \includegraphics[width=0.63\linewidth]{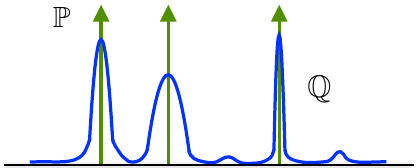}
    \caption{For appropriate $c$ and $\varepsilon$, $\mathbb Q$ is an $\varepsilon$-perturbation of the empirical distribution $\P$, and belongs to the ambiguity set $\ball{\varepsilon}{c}{\P}$.}
    \label{fig:Q}
\end{figure}

In this section, we formally define \gls{acr:ot} ambiguity sets and explore their modeling power for uncertainty quantification. Consider a non-negative lower semi-continuous function $c:\mathcal X\times \mathcal X\to\nonnegativeReals$ (henceforth, referred to as \emph{transportation cost}) and two probability distributions $\P,\Q\in\probSpace{\mathcal X}$. Then, the \emph{\gls{acr:ot} discrepancy} between $\P$ and $\Q$ is defined by  
\begin{equation}\label{eq:otcost}
    \transportCost{c}{\P}{\Q}
    \coloneqq 
    \inf_{\gamma\in\setPlans{\P}{\Q}}\int_{\mathcal X\times \mathcal X}c(x_1,x_2)\d\gamma(x_1,x_2),
\end{equation}
where $\setPlans{\P}{\Q}$ is the set of all probability distributions over $\mathcal X\times\mathcal X$ with marginals $\P$ and $\Q$, often called \emph{transport plans} or \emph{couplings}~\cite{Villani2009a}. The semantics are as follows: we seek the minimum cost to transport the probability distribution $\P$ onto the probability distribution $\Q$ when transporting a unit of mass from $x_1$ to $x_2$ costs $c(x_1,x_2)$. If $c(x_1,x_2) = \|x_1-x_2\|^p$, for some norm $\|\cdot\|$ on $\reals^d$, $(W_c)^{1/p}$ boils down to the celebrated (type-$p$) \emph{Wasserstein} distance. Intuitively, $\transportCost{c}{\P}{\Q}$ quantifies the discrepancy between $\P$ and $\Q$, and it naturally provides us with a definition of ambiguity in the space of probability distributions (see, e.g., \cref{fig:Q}). In particular, the \emph{\gls{acr:ot} ambiguity set} of radius $\varepsilon$ and centered at $\P$ is defined as 
\begin{equation}\label{eq:otambiguityset}
    \ball{\varepsilon}{c}{\P}
    \coloneqq 
    \{\Q\in\probSpace{\mathcal X}: \transportCost{c}{\P}{\Q}\leq\varepsilon\}
    \subset\probSpace{\mathcal X}.
\end{equation}
In words, $\ball{\varepsilon}{c}{\P}$ contains all probability distributions onto which $\P$ can be transported with a budget of at most $\varepsilon$. This includes both continuous and discrete distributions, distributions not concentrated on the support of $\P$, and even distributions whose mass asymptotically escapes to infinity, as shown next.
\begin{example}
Let $\mathcal X=\reals$, $c(x_1,x_2)=|x_1-x_2|^2$, $\P=\delta_0$, and let $\Q$ be the Gaussian distribution with mean $0$ and variance $\varepsilon$. Then, $\transportCost{c}{\diracDelta{0}}{\Q} = \mathbb E_{\Q}\left[|x|^2  \right]=\varepsilon$. Moreover, $\transportCost{c}{\diracDelta{0}}{\diracDelta{\sqrt{\varepsilon}}}=\varepsilon$, and $\transportCost{c}{\diracDelta{0}}{{\frac{\varepsilon}{n^2}}\diracDelta{n}+(1-{\frac{\varepsilon}{n^2}})\diracDelta{0}}=\varepsilon$ for all $n\in\naturals_{\geq\sqrt{\varepsilon}}$.
\end{example}
These properties cease to hold if the discrepancy between probability distributions is measured via the Kullback-Leibler (KL) divergence or Total Variation (TV) distance \cite{gibbs2002choosing}. We conclude this subsection by highlighting that \gls{acr:ot} ambiguity sets are well-behaved under monotone changes in $c$ and $\varepsilon$.
\begin{lemma}
\label{lemma:ambiguitysetmonotone}
Let $\P\in\probSpace{\mathcal X}$, $c, c_1, c_2$ be transportation costs over $\mathcal X \times \mathcal X$, and $\varepsilon,\varepsilon_1, \varepsilon_2>0$.
\begin{enumerate}
    \item If $\varepsilon_1\leq \varepsilon_2$ then $\ball{\varepsilon_1}{c}{\P}\subseteq\ball{\varepsilon_2}{c}{\P}$;
    \item If $c_1\leq c_2$, then $\ball{\varepsilon}{c_2}{\P}\subseteq\ball{\varepsilon}{c_1}{\P}$.
\end{enumerate}
\end{lemma}
In words, an increase in the transportation cost shrinks the \gls{acr:ot} ambiguity set, whereas an increase of the radius enlarges it. These simple observations arm practitioners with actionable knobs to control the level of distributional uncertainty.

% --------------------------------------------------------------------

%\subsection{Expressivity of \gls{acr:ot} ambiguity sets}
\subsection{Robustness in (space, likelihood)}
\label{subsec:properties:OT}

\begin{figure*}[t]
    \centering
    \includegraphics[width=0.99\textwidth]{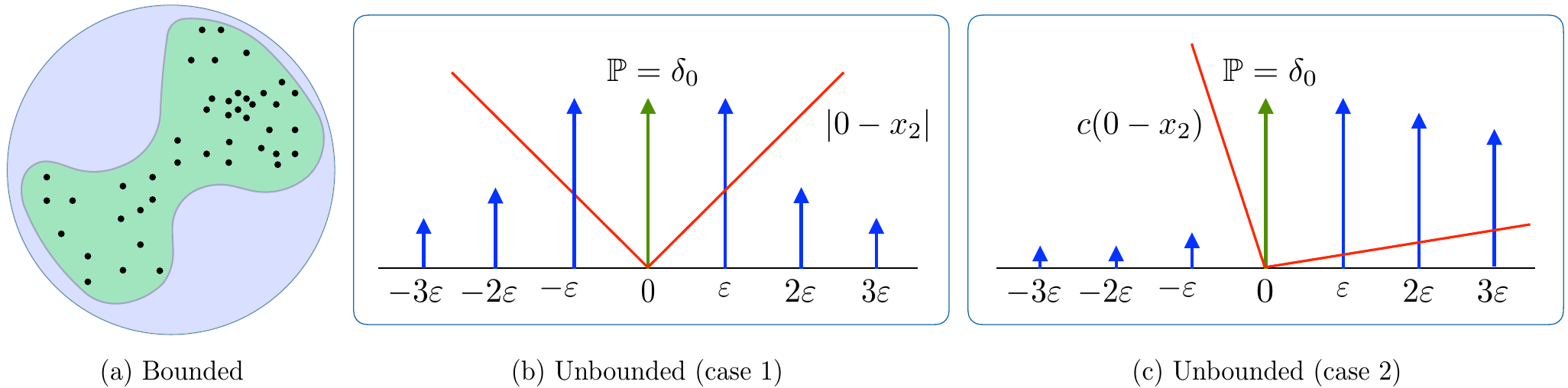}
    \caption{\gls{acr:ot} ambiguity sets allow to: (a) reduce conservativeness of robust uncertainty models, (b) capture arbitrarily large uncertainty values, together with the information about how rare those events are, and (c) include side information about the uncertainty in the transportation cost.}
    \label{fig:AmbSet}
\end{figure*}

\gls{acr:ot} ambiguity sets are attractive to model distributional uncertainty for various reasons, which we detail next. 
%\smallskip
%\textbf{Robustness in (\emph{space}, \emph{likelihood})}.
Differently from standard uncertainty models, namely robust (i.e., norm-bounded uncertainty) and stochastic (i.e., one fixed distribution), \gls{acr:ot} ambiguity sets capture uncertainty in both space (i.e., the support of the distribution) and likelihood (i.e., the probability of specific events). These are encoded in $\ball{\varepsilon}{c}{\P}$ through $c$ and $\varepsilon$: given a reference distribution $\P$ and a radius $\varepsilon$ (intended as transportation `budget'), the transportation cost $c$ controls the displacement of probability mass from $\P$, and consequently the \emph{shape} and \emph{size} of $\ball{\varepsilon}{c}{\P}$. Specifically, given $x_1$ in the support of $\P$, then higher the value of $c(x_1,x_2)$, less probability mass is transported from $x_1$ to $x_2$. This automatically bounds the likelihood of the value $x_2$ for any distribution in the \gls{acr:ot} ambiguity set $\ball{\varepsilon}{c}{\P}$. In particular, if $c(x_1,x_2)$ is large enough for all $x_1$ in the support of $\P$ and all $x_2 \in \mathcal A$, for some $\mathcal A \subseteq \mathcal X$, then no distribution inside $\ball{\varepsilon}{c}{\P}$ will be supported on $\mathcal A$. 

This feature of \gls{acr:ot} ambiguity sets allows us to capture a great variety of uncertainty phenomena, as detailed next. %Before proceeding, recall that $\ball{\varepsilon}{c}{\P}$ is defined over an arbitrary measurable set $\mathcal X$. 

\smallskip

\emph{(i) $\mathcal X$ bounded.}
In this case, $\ball{\varepsilon}{c}{\P}$ naturally interpolates between the robust and the stochastic uncertainty models. Specifically, if $\varepsilon \to \text{diam}(\mathcal X) =\max_{x_1,x_2 \in \mathcal X} c(x_1,x_2)$, then $\ball{\varepsilon}{c}{\P}$ recovers the probabilistic representation of the robust model  (i.e., the set of all delta distributions $\delta_\xi$ for all $\xi \in \mathcal X$). If instead $\varepsilon \to 0$, then $\ball{\varepsilon}{c}{\P} = \P$ recovers the stochastic model. This interpolation between the two models allows to \emph{reduce the conservativeness} of the robust model, while \emph{robustifying} the stochastic model. Importantly, this property opens the door to studying decision-making formulations which optimally trade between safety and performance (see \cite{aolaritei2023wasserstein} for a control application). This is particularly important in \emph{data-driven} scenarios, where $\P$ is an empirical distribution supported on a finite number of samples from the uncertainty. This is illustrated in \cref{fig:AmbSet}(a), where the black dots are the samples, the green shape is the support of the true unknown uncertainty distribution, and the blue circle is the robust model. In this case, $\ball{\varepsilon}{c}{\P}$ can be properly designed to capture with high probability only the green area.

\smallskip

\emph{(ii) $\mathcal X$ unbounded}. In this case, \gls{acr:ot} ambiguity sets have the ability to capture arbitrarily large uncertainty values, together with the information about how rare those events are. We illustrate this reasoning in \cref{fig:AmbSet}(b), where for simplicity of exposition we consider $\mathcal X=\reals$, $c(x_1,x_2)=|x_1-x_2|$ (represented by the red lines), $\P = \delta_0$ (i.e., the reference distribution is empirical and supported on the point $0$). Then, a simple calculation shows that $\transportCost{c}{\diracDelta{0}}{\frac{t-1}{t}\diracDelta{0} + \frac{1}{t}\diracDelta{t \varepsilon}}=\varepsilon$ for all $t\in\mathbb Z$. In words, $\ball{\varepsilon}{c}{\P}\coloneqq\ball{\varepsilon}{|\cdot|}{\delta_0}$ captures the uncertainty value $t \varepsilon$ together with the information that the probability of this value is at most $\frac{1}{t}$. This reasoning extends to more general scenarios (e.g., arbitrary $\P$ and general $c$; see \cref{fig:AmbSet}(c)) and makes \gls{acr:ot} ambiguity sets particularly suitable to model and robustify against \emph{black swan events}, i.e., extreme outliers with catastrophic effects which are impossibly difficult to predict. This is experimentally validated in \cite{aolaritei2023hedging}, in the context of day-ahead energy trading in energy markets.

% --------------------------------------------------------------------
% --------------------------------------------------------------------
% --------------------------------------------------------------------

\section{Pushforward of \gls{acr:ot} Ambiguity Sets}
\label{sec:propagation}

In this section, we study how \gls{acr:ot} ambiguity sets propagate via linear and nonlinear transformations. We start by showing that naive approaches (in particular, propagation of the center only, or propagation based on Lipschitz bounds) fail to effectively capture the propagation of distributional uncertainty.

% --------------------------------------------------------------------

\subsection{Naive Approaches and their Shortcomings}

Given the \gls{acr:ot} ambiguity set $\ball{\varepsilon}{c}{\P}$, one might be tempted to approximate the result of the propagation $\pushforward{f}\ball{\varepsilon}{c}{\P}$ by $\ball{\varepsilon}{c}{\pushforward{f}\P}$. This approach suffers from fundamental limitations already in very simple settings, easily resulting in crude overestimation or catastrophic underestimation of the distributional uncertainty, as shown in the next example.

\begin{example}
\label{ex:propagation:centernotenough}
Let $c(x_1-x_2)=|x_1-x_2|^2$ on $\reals$.
\begin{enumerate}
    \item Let $f \equiv 0$. Then, $\pushforward{f}\ball{\varepsilon}{c}{\P}$ only contains $\diracDelta{0}$ (regardless of $\P$), whereas $\ball{\varepsilon}{c}{\pushforward{f}\P}=\ball{\varepsilon}{c}{\diracDelta{0}}$ contains all distributions whose second moment is at most $\varepsilon$. Therefore, $\ball{\varepsilon}{c}{\pushforward{f}\P}$ \emph{overestimates} the true distributional uncertainty.
    
    \item Let $f(x)=2x$, and $\P = \delta_0$. Then, $\ball{\varepsilon}{c}{\pushforward{f}\P}=\ball{\varepsilon}{c}{\diracDelta{0}}$. In~\cref{thm:nonlin:trans} we show that $\pushforward{f}\ball{\varepsilon}{c}{\P} = \ball{4\varepsilon}{c}{\diracDelta{0}}$. Thus, $\ball{\varepsilon}{c}{\pushforward{f}\P}$ \emph{underestimates} the true uncertainty.
\end{enumerate}
\end{example}

Moreover, one might be tempted to bound the propagated distributional uncertainty with the Lipschitz constant $L$ of $f$, i.e., to upper bound $\pushforward{f}\ball{\varepsilon}{c}{\P}$ with $\ball{L\varepsilon}{c}{f_\# \P}$.
However, this approach suffers from three major limitations.
First, many transformations are not Lipschitz continuous.
Second, the transportation cost $c(x_1,x_2)$ might not be $\norm{x_1-x_2}$, which makes the Lipschitz bound not directly applicable. Indeed, already in~\cref{ex:propagation:centernotenough}, an increase by a factor of 2 in the radius does not alleviate the underestimation of the true ambiguity set: one needs to use $L^2=2^2$ to account for the quadratic transportation cost $|x_1-x_2|^2$.
Finally, even if the transportation cost is $\norm{x_1-x_2}$, Lipschitz bounds might be overly conservative, as shown next. 
 
\begin{example}
Let $c(x_1,x_2)=\|x_1-x_2\|_2$ on $\reals^n$, $\P = \delta_0$, and $A$ be a diagonal matrix with diagonal entries $\{n,0,\ldots,0\}$ (thus $L =n$). Then, $\pushforward{A}\ball{\varepsilon}{c}{\delta_0}$ eliminates all the distributional uncertainty in the last $d-1$ dimensions. However, $\ball{n\varepsilon}{c}{\delta_0}$ contains, among others, all distributions of the form $\delta_0 \otimes \Q$ with $\delta_0 \in \mathcal P(\reals)$, and $\Q \in \mathcal P(\reals^{n-1})$ satisfying $\mathbb E_{\Q}\left[\|x\|_2  \right]\leq n\varepsilon$.
\end{example}

These shortcomings prompt us to study the exact propagation of \gls{acr:ot} ambiguity sets.

% --------------------------------------------------------------------

\subsection{Nonlinear Transformations}
\label{subsec:nonlin:trans}

We seek to characterize the ambiguity set $\pushforward{f}\ball{\varepsilon}{c}{\P}$ for a general Borel map $f:\mathcal{X}\to\mathcal{Y}$. We proceed in three steps. In~\cref{lemma:stability set plans}, we prove a result at the level of the \emph{set of couplings} between two probability distributions $\P$ and $\Q$. Armed with this result, in \cref{prop:Wc:pushforward} we study the pushforward at the level of the \emph{\gls{acr:ot} discrepancy}. Our analysis culminates in~\cref{thm:nonlin:trans}, where we show how general \gls{acr:ot} ambiguity sets propagate under pushforward via arbitrary (nonlinear) maps.

At the level of couplings, the propagation is well behaved: the set of couplings between $\pushforward{f}\P$ and $\pushforward{f}\Q$, i.e., $\setPlans{\pushforward{f}\P}{\pushforward{f}\Q}$, corresponds to the pushforward via $\tensorProd{f}{f}$ of $\setPlans{\P}{\Q}$.

\begin{lemma}
\label{lemma:stability set plans}
Let $\P, \Q\in\probSpace{\mathcal X}$, and consider an arbitrary transformation $f:\mathcal X \to \mathcal Y$. Then,
\begin{equation*}
    \pushforward{(\tensorProd{f}{f})}\setPlans{\P}{\Q}
    =\setPlans{\pushforward{f}\P}{\pushforward{f}\Q}.
\end{equation*}
\end{lemma}

We now study the \gls{acr:ot} discrepancy \eqref{eq:otcost} between the pushforward of two probability distributions. \cref{lemma:stability set plans} stipulates that, instead of optimizing over the set of all couplings $\setPlans{\pushforward{f}\P}{\pushforward{f}\Q}$, we can equivalently optimize over the pushforward of all couplings $\pushforward{(\tensorProd{f}{f})}\setPlans{\P}{\Q}$.
This helps us shed light on the relation between the \gls{acr:ot} discrepancy between $\pushforward{f}\P$ and $\pushforward{f}\Q$ and the \gls{acr:ot} discrepancy between $\P$ and $\Q$.

\begin{proposition}
\label{prop:Wc:pushforward}
Let $\P, \Q\in\probSpace{\mathcal X}$, consider a transformation $f:\mathcal X \to \mathcal Y$, and let $d:\mathcal Y \times \mathcal Y \to \nonnegativeReals$ be a transportation cost on $\mathcal Y$. Then,
    \begin{align}
    \label{eq:W_c:arbitrary:f}
        \transportCost{d\circ (\tensorProd{f}{f})}{\P}{\Q}
        =
        \transportCost{d}{\pushforward{f}\P}{\pushforward{f}\Q}.
    \end{align}
Moreover, let $c:\mathcal X \times \mathcal X \to \nonnegativeReals$ be a transportation cost on $\mathcal X$. Then, we have the following cases:

1) If $f$ is bijective, with inverse $\inv{f}$, then
    \begin{align}
    \label{eq:W_c:bijective:f}
        \transportCost{c}{\P}{\Q}
        =
        \transportCost{c\circ (\tensorProd{\inv{f}}{\inv{f}})}{\pushforward{f}\P}{\pushforward{f}\Q}.
    \end{align}
    
2) If $f$ is injective, with left inverse $\inv{f}$, then
    \begin{align}
    \label{eq:W_c:injective:f}
        \transportCost{c}{\P}{\Q}
        =
        \transportCost{c\circ (\tensorProd{\inv{f}}{\inv{f}})}{\pushforward{f}\P}{\pushforward{f}\Q}.
    \end{align}
    
3) If $f$ is surjective, with right inverse $\inv{f}$, then 
    \begin{align}
    \label{eq:W_c:surjective:f}
        \transportCost{c\circ (\tensorProd{\inv{f}}{\inv{f}}) \circ (\tensorProd{f}{f})}{\P}{\Q} = \transportCost{c\circ (\tensorProd{\inv{f}}{\inv{f}})}{\pushforward{f}\P}{\pushforward{f}\Q}.
    \end{align} 
\end{proposition}

\cref{prop:Wc:pushforward} states that if the transportation cost $c$ on $\mathcal{X}$ is of the form $d\circ(\tensorProd{f}{f})$, for some transportation cost $d$ on $\mathcal{Y}$, then the \gls{acr:ot} discrepancies $\transportCost{c}{\P}{\Q}$ and $\transportCost{d}{\pushforward{f}\P}{\pushforward{f}\Q}$ coincide. 
This a priori assumption of the transportation cost can be dropped if the function $f$ satisfies additional assumptions.  
For instance, if $f$ is bijective or injective (and thus, a left inverse $\inv{f}$ exists), then for \emph{any} transportation cost $c$ on $\mathcal{X}$, the \gls{acr:ot} discrepancies $\transportCost{c}{\P}{\Q}$ and  $\transportCost{c\circ(\tensorProd{\inv{f}}{\inv{f}})}{\pushforward{f}\P}{\pushforward{f}\Q}$ coincide. 
We can now state the main result of this subsection.

\begin{figure}[t]
    \centering
    \includegraphics[width=0.75\linewidth]{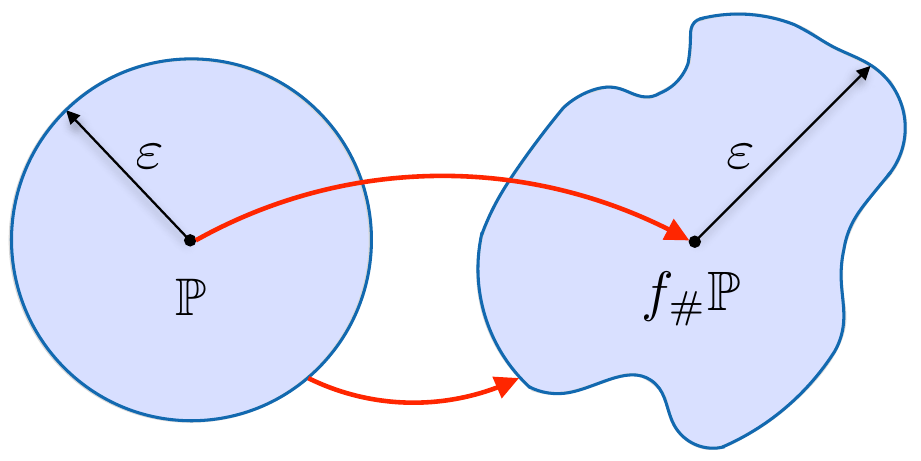}
    \caption{Pushforward of \gls{acr:ot} ambiguity sets.}
    \label{fig:prop}
\end{figure}

\begin{theorem}[Nonlinear transformations]
\label{thm:nonlin:trans}
Let $\P\in\probSpace{\mathcal X}$, consider an arbitrary transformation $f:\mathcal X \to \mathcal Y$, and let $d:\mathcal Y \times \mathcal Y \to \nonnegativeReals$ be a transportation cost on $\mathcal Y$. Then,
    \begin{align}
    \label{eq:B_epsilon:arbitrary:f}
        \pushforward{f}{\mathbb B_\varepsilon^{d\circ (\tensorProd{f}{f})}(\P)} 
        \subset 
        \mathbb B_\varepsilon^{d}(\pushforward{f}{\P}).
    \end{align}
Moreover, let $c:\mathcal X \times \mathcal X \to \nonnegativeReals$ be a transportation cost on $\mathcal X$. Then, we have the following cases:
\begin{enumerate}
\item If $f$ is bijective, with inverse $\inv{f}$, then
    \begin{align}
    \label{eq:B_epsilon:bijective:f}
        \pushforward{f}{\mathbb B_\varepsilon^{c}(\P)} 
        = 
        \mathbb B_\varepsilon^{c\circ (\tensorProd{\inv{f}}{\inv{f}})}(\pushforward{f}{\P}).
    \end{align}
    
\item If $f$ is injective, with left inverse $\inv{f}$, then
    \begin{equation}
    \begin{aligned}
    \label{eq:B_epsilon:injective:f}
        \pushforward{f}{\mathbb B_\varepsilon^{c}(\P)} 
        &= 
        \pushforward{(f \circ \inv{f})}\mathbb B_\varepsilon^{c\circ (\tensorProd{\inv{f}}{\inv{f}})}(\pushforward{f}{\P}) 
        \\
        &\subset 
        \mathbb B_\varepsilon^{c\circ (\tensorProd{\inv{f}}{\inv{f}})}(\pushforward{f}{\P}).
    \end{aligned}
    \end{equation}
    
\item If $f$ is surjective, with right inverse $\inv{f}$, then 
    \begin{equation}
    \begin{aligned}
    \label{eq:B_epsilon:surjective:f}
        \pushforward{f}{\ball{\varepsilon}{c\circ (\tensorProd{\inv{f}}{\inv{f}}) \circ (\tensorProd{f}{f})}{\P}} 
        &= 
        \mathbb B_\varepsilon^{c\circ (\tensorProd{\inv{f}}{\inv{f}})}(\pushforward{f}{\P}) 
        \\
        &\subset 
        \pushforward{f}{\mathbb B_\varepsilon^{c}\left(\pushforward{(\inv{f}\circ f)}\P\right)}.
    \end{aligned}
    \end{equation} 
\end{enumerate}
\end{theorem}

\begin{remark}
We highlight the striking similarity between the propagation of OT ambiguity sets in Theorem~\ref{thm:nonlin:trans} and the propagation of deterministic bounded sets defined as $\mathcal{B}_\varepsilon^c(x_0)\coloneqq\{x\in\mathcal X: c(x_0,x)\leq\varepsilon\}$. Simple calculations show that for a bijective transformations $f:\mathcal X \to \mathcal Y$ we have $f(\mathcal{B}_\varepsilon^c(x_0))
=\{y\in \mathcal Y:c(f^{-1}(f(x_0)),f^{-1}(y))\leq \varepsilon\} =\mathcal{B}_\varepsilon^{c\circ (f^{-1}\times f^{-1}})(f(x_0))$. Surprisingly, an analogous result, namely~\eqref{eq:B_epsilon:bijective:f}, holds in probability spaces for~\gls{acr:ot} ambiguity sets. This shows that propagating OT ambiguity sets is \emph{as easy as} propagating standard deterministic sets.
%The results in~\cref{thm:nonlin:trans} are reminiscent of the propagation of an ambiguity set $\mathcal{B}_\varepsilon^c(x_0)\coloneqq\{x\in\mathcal X: c(x_0,x)\leq\varepsilon\}$ on $\mathcal X$ through a nonlinear map $f$. For instance, consider the bijective case: $f(\mathcal{B}_\varepsilon^c(x_0))
%=\{y\in Y:c(x_0,f^{-1})(y)\leq \varepsilon\} =\{y\in Y:c(f^{-1}(f(x_0)),f^{-1}(y))\leq \varepsilon\} =\mathcal{B}_\varepsilon^{c\circ (f^{-1}\times f^{-1}})(f(x_0))$. Surprisingly, an analogous result, namely~\eqref{eq:B_epsilon:bijective:f}, holds in probability spaces for~\gls{acr:ot} ambiguity sets. 
\end{remark}

\begin{remark}
In general, inclusion~\eqref{eq:B_epsilon:injective:f} cannot be improved to equality. For instance, consider the map $f:\reals\to\reals^2$ defined by $f(x_1)=(x_1,0)$, with left inverse $\inv{f}(x_1,x_2)=x_1$. Clearly, $f$ is injective. Moreover, consider the transportation cost $c(x_1,x_2)=|x_1-x_2|$, $\P=\delta_{0}$. Let $\Q=\delta_{(0,1)}$. By construction, $\Q$ cannot be obtained from the pushforward via $f$ of a probability distribution over $\reals$, and so $\Q\notin\pushforward{f}\ball{\varepsilon}{c}{\P}$. However, $\transportCost{c\circ(\tensorProd{\inv{f}}{\inv{f}})}{\Q}{\pushforward{f}\P}
    =
    |\inv{f}(0,1)-\inv{f}(0,0)|=0$,
and so $\Q\in\ball{\varepsilon}{c\circ(\tensorProd{\inv{f}}{\inv{f}})}{\pushforward{f}\P}$.
\end{remark}

In a nutshell, \cref{thm:nonlin:trans} asserts that the pushforward of an \gls{acr:ot} ambiguity set of radius $\varepsilon$ centered at $\P$ via an invertible map $f$ coincides with an \gls{acr:ot} ambiguity set with the same radius $\varepsilon$ centered at $\pushforward{f}\P$, constructed with the transportation cost $c\circ(\tensorProd{\inv{f}}{\inv{f}})$; see \cref{fig:prop}. When $f$ is only injective, the equality reduces to inclusion. In the surjective case, an upper bound for $\pushforward{f}\ball{\varepsilon}{c}{\P}$ can be easily recovered from \cref{lemma:ambiguitysetmonotone}:

\begin{corollary}
\label{cor:nonlin:trans}
Let $\P\in\probSpace{\mathcal X}$, consider an arbitrary surjective transformation $f:\mathcal X\to \mathcal Y$, with right inverse $\inv{f}$, and let $c:\mathcal X\times \mathcal X\to\nonnegativeReals$ be a transportation cost on $\mathcal X$ satisfying
\begin{equation}
    \label{eq:B_epsilon:surjective:f:rivisited:assumption}
    (c\circ(\tensorProd{\inv{f}}{\inv{f}})\circ(\tensorProd{f}{f}))(x_1,x_2)
    \leq
    c(x_1,x_2)
\end{equation}
for all $x_1,x_2\in \mathcal X$.
Then, 
\begin{equation}
    \label{eq:B_epsilon:surjective:f:rivisited}
    \pushforward{f}\ball{\varepsilon}{c}{\P}
    \subset
    \ball{\varepsilon}{c\circ(\tensorProd{\inv{f}}{\inv{f}})}{\pushforward{f}\P}.
\end{equation}
If equality holds in~\eqref{eq:B_epsilon:surjective:f:rivisited:assumption}, then equality holds in~\eqref{eq:B_epsilon:surjective:f:rivisited}.
\end{corollary}

We exemplify~\cref{cor:nonlin:trans} in the case of norms. 

\begin{example}
\label{example:surjective:norm:1}
Let $f:\reals^n\to\nonnegativeReals$ be defined by $f(x)=\norm{x}$, where $\norm{\cdot}$ is a norm on $\reals^n$. Clearly, $f$ is surjective, and a right inverse is given by $\inv{f}(t)=tx_0$ for some $x_0\in\reals^n$ with $\norm{x_0}=1$. Then,  $(\inv{f}\circ f)(x)=\inv{f}(\norm{x})=\norm{x}x_0$.
Consider a transportation cost $c(x_1,x_2)=\psi(\norm{x_1-x_2})$, for monotone $\psi:\nonnegativeReals\to\nonnegativeReals$. Elementary calculations show that $(c\circ(\tensorProd{\inv{f}}{\inv{f}})\circ(\tensorProd{f}{f}))(x_1,x_2)\leq c(x_1,x_2)$.
\begin{comment}
\begin{equation*}
\begin{aligned}
    (c\circ(\tensorProd{\inv{f}}{\inv{f}})\circ(\tensorProd{f}{f}))(x_1,x_2)
    &=
    c(\norm{x_1}x_0,\norm{x_2}x_0)
    \\
    &=\psi(\norm{\norm{x_1}x_0-\norm{x_2}x_0})
    \\
    &=
    \psi(|\norm{x_1}-\norm{x_2}|)
    \\
    &\leq 
    \psi(\norm{x_1-x_2})
    \\
    &=
    c(x_1,x_2),
\end{aligned}
\end{equation*}
where the inequality follows from $|\norm{x_1}-\norm{x_2}|\leq \norm{x_1-x_2}$ together with the monotonicity of $\psi$.
\end{comment}
Therefore, \cref{cor:nonlin:trans} yields the inclusion $\pushforward{f}\ball{\varepsilon}{c}{\P}\subset\ball{\varepsilon}{c\circ(\tensorProd{\inv{f}}{\inv{f}})}{\pushforward{f}\P}$.
\end{example}

% --------------------------------------------------------------------

\subsection{Linear Transformations}
\label{subsec:lin:trans}

We now investigate \emph{linear} transformations between finite-dimensional Euclidean spaces, defined by a matrix $A\in\reals^{m\times n}$. The linear structure, together with mild assumptions on the transportation cost, allow us to strengthen the results of~\cref{subsec:nonlin:trans} in three ways.
First, we will derive an upper bound on $\pushforward{A}\ball{\varepsilon}{c}{\P}$ even if the linear map defined by $A$ is not injective or surjective. In such case, right and left inverses do not exist, and \cref{thm:nonlin:trans} is not applicable.
Second, we will prove a chain of inclusions that strenghtens the statements of~\cref{thm:nonlin:trans} (when restricted to linear transformations). Finally, we will show that the equality \eqref{eq:B_epsilon:bijective:f} can be extended to arbitrary surjective linear maps.

To accomplish this, we restrict the analysis to finite-dimensional Hilbert spaces. Consider $\mathcal X\coloneqq (\reals^n, \innerProduct{\cdot}{\cdot}_X)$, with $\innerProduct{x_1}{x_2}_X = x_1^\top X x_2$, for $X \in \reals^{n \times n}$ symmetric and positive definite, and $\mathcal Y\coloneqq  (\reals^m, \innerProduct{\cdot}{\cdot}_Y)$, with $Y \in \reals^{m \times m}$ symmetric and positive definite. Given $A\in\reals^{m\times n}$, let $\adjoint{A}$ denote its adjoint and $\pinv{A}$ denote its Moore-Penrose inverse.% \cite{penrose1955generalized}. 

\begin{remark}
Following the definition of $\mathcal X$ and $\mathcal Y$, it can be easily shown that $\adjoint{A}=X^{-1} \T{A} Y$. In particular, if $X=\eye_{n \times n}$ and $Y=\eye_{m \times m}$, we recover the standard inner product and $\adjoint{A}=\T{A}$ (i.e, the adjoint is the transpose). Moreover, if $A$ is full rank, then $\pinv{A}$ has a closed-form expression. Specifically, if $A$ is full row-rank then $\pinv{A}=\inv{X}\T{A}\inv{(A\inv{X}\T{A})}$, and if $A$ is full column-rank then $\pinv{A}=\inv{(\T{A}Y A)}\T{A}Y$.
\end{remark}

Throughout this subsection, we require the following assumption on the transportation cost.

\begin{assumption}
\label{assump:transportation:cost}
The transportation cost $c$ is
\begin{itemize}
    \item[(i)] translation-invariant: $c(x_1,x_2) := c(x_1 - x_2)$;

    \item[(ii)] orthomonotone: $c(x_1+x_2) \geq c(x_1)$ for all $x_1,x_2 \in \reals^n$ satisfying $\innerProduct{x_1}{x_2}_X = 0$.
\end{itemize}
\end{assumption}

Following Assumption~\ref{assump:transportation:cost}, the transportation cost $c$ becomes a function $\reals^n \to \nonnegativeReals$. Accordingly, we use the simplified notation $c \circ A^{-1}$ instead of $c \circ (A^{-1} \times A^{-1})$; i.e., $(c \circ A^{-1})(x_1-x_2) = c (A^{-1}x_1 - A^{-1}x_2)$.

\begin{remark}
All transportation costs of the form $c(x_1-x_2)=\psi(\norm{x_1-x_2}_X)$, where $\norm{\cdot}_X$ is the norm induced by the inner product $ \innerProduct{\cdot}{\cdot}_X$, and $\psi:\nonnegativeReals\to\nonnegativeReals$ is monotone and lower semi-continuous, are translation-invariant and orthomonotone. Indeed, for all $x_1,x_2\in\reals^n$ such that $\innerProduct{x_1}{x_2}_X=0$,
\begin{equation*}
    \psi(\norm{x_1+x_2}_X)
    =\psi\left(\sqrt{\norm{x_1}_X^2+\norm{x_2}_X^2}\right)
    \geq \psi(\norm{x_1}_X).
\end{equation*}
\end{remark}

We are now ready to state the main result of this subsection.

\begin{theorem}[Linear transformations]
\label{thm:lin:trans}
Let $\P\in\probSpace{\reals^n}$ and $A \in \reals^{m \times n}$. Moreover, let $c:\reals^n \to \nonnegativeReals$ satisfy \cref{assump:transportation:cost}. Then, the following chain of relationships holds:
\begin{equation}
\begin{aligned}
\label{eq:B_epsilon:arbitrary:A}
    \pushforward{A}{\ball{\varepsilon}{c}{\P}} 
    &\subset
    \pushforward{(A \pinv{A})} \ball{\varepsilon}{c \circ \pinv{A}}{\pushforward{A}\P}
    \\
    &=
    \pushforward{A} \ball{\varepsilon}{c}{\pushforward{(\pinv{A}A)} \P}
    \subset 
    \ball{\varepsilon}{c \circ \pinv{A}}{\pushforward{A}\P},
\end{aligned}
\end{equation}
Moreover, if the matrix $A$ is full row-rank (i.e., surjective), then all the inclusions in \eqref{eq:B_epsilon:arbitrary:A} collapse to equality, and we have
\begin{align}
\label{eq:B_epsilon:surjective:A:1}
    \pushforward{A}\ball{\varepsilon}{c}{\P} 
    = 
    \ball{\varepsilon}{c \circ \pinv{A}}{\pushforward{A}\P}.
\end{align}
%with $\pinv{A}=\inv{X}\T{A}\inv{(A\inv{X}\T{A})}$.
\end{theorem}

In words, \cref{thm:lin:trans} asserts that the result of the propagation $\pushforward{A}{\ball{\varepsilon}{c}{\P}}$ is itself an \gls{acr:ot} ambiguity set (or can be upper bounded by one), with the same radius $\varepsilon$, propagated center $\pushforward{A}\P$, and an $A$-induced transportation cost $c\circ \pinv{A}$. 
The following example shows that the equality \eqref{eq:B_epsilon:surjective:A:1} does generally \emph{not} hold for non-surjective linear maps.
\begin{example}
Consider $X = Y = \eye_{2 \times 2}$, and
$
A\coloneqq
\begin{bsmallmatrix}
    1 & 0 \\ 0 & 0
\end{bsmallmatrix}
$
with pseudoinverse $\pinv{A}=A$, the quadratic transportation cost $c(x_1-x_2)=\norm{x_1-x_2}_2^2$, and the probability distribution $\P=\delta_{(0,0)}\in\mathcal{P}(\reals^2)$. Let $\Q=\delta_{(0,1)}\in\mathcal{P}(\reals^2)$. Since $(0,1)\not\in\range(A)$, $\Q$ does not belong to $\pushforward{A}\ball{\varepsilon}{c}{\P}$ regardless of $\varepsilon$. However, $\transportCost{c\circ\pinv{A}}{\Q}{\pushforward{A}\P}=\norm{\pinv{A}\begin{bsmallmatrix} 0 \\ 0\end{bsmallmatrix}-\pinv{A}\begin{bsmallmatrix} 0 \\ 1\end{bsmallmatrix}}_2^2=0$.
Thus, $\Q\in\ball{\varepsilon}{c \circ \pinv{A}}{\pushforward{A}\P}$, and 
$\pushforward{A}\ball{\varepsilon}{c}{\P}\subsetneq\ball{\varepsilon}{c \circ \pinv{A} }{\pushforward{A}\P}$.
\end{example}

\cref{thm:lin:trans} continues to hold if the \gls{acr:ot} ambiguity set $\ball{\varepsilon}{c}{\P}$ is defined over a subset $\mathcal X \subset \reals^n$ (i.e., $\ball{\varepsilon}{c}{\P}$ contains only distributions supported on $\mathcal X$). This is summarized in the next corollary, whose proof follows similar lines to the proof of \cref{thm:lin:trans} and is therefore omitted. 

\begin{corollary}
\label{cor:lin:trans:bounded}
Let $\mathcal X \subset \reals^n$, $\P \in \probSpace{\mathcal X}$ and let $\ball{\varepsilon}{c}{\P}$ be defined over $\mathcal X$. Moreover, let $A \in \reals^{m \times n}$ be full-row rank and let $c:\reals^n \to \nonnegativeReals$ satisfy \cref{assump:transportation:cost}. Then,
$
    \pushforward{A}\ball{\varepsilon}{c}{\P} 
    = 
    \ball{\varepsilon}{c \circ \pinv{A}}{\pushforward{A}\P},
$
with $\ball{\varepsilon}{c \circ \pinv{A}}{\pushforward{A}\P}$ restricted to all distributions supported on $A \mathcal X  \subseteq \reals^m$.
\end{corollary}

\cref{thm:nonlin:trans,thm:lin:trans} provide a full description for the case of full-rank matrices. For invertible and full column-rank matrices (i.e., the injective case), the result follows directly from~\cref{thm:nonlin:trans} with no assumption on the transportation cost (i.e., for a \emph{general} $c:\reals^n\times\reals^n\to\nonnegativeReals$). The surjective case, instead, relies on~\cref{thm:lin:trans}, and therefore requires the transportation cost to be translation-invariant and orthomonotone.

\begin{comment}
\begin{corollary}
\label{cor:lin:trans:fullrank}
Let $\P\in\probSpace{\reals^n}$, $c:\reals^n\times\reals^n\to\nonnegativeReals$ be a general transportation cost, and $\varepsilon>0$.
\begin{enumerate}
    \item If $A \in \reals^{n \times n}$ is non-singular with inverse $\inv{A}$, then
    \begin{align}
    \label{eq:B_epsilon:bijective:A}
        \pushforward{A}{\ball{\varepsilon}{c}{\P}}
        =
        \ball{\varepsilon}{c \circ(\inv{A} \times \inv{A})}{\pushforward{A}\P}.
    \end{align}
    
    \item If $A \in \reals^{m \times n}$ has full column-rank, then
    \begin{equation}\label{eq:B_epsilon:injective:A}
    \begin{aligned}
        \pushforward{A}\ball{\varepsilon}{c}{\P}
        &=
        \pushforward{(A \pinv{A})} \ball{\varepsilon}{c \circ (\pinv{A} \times \pinv{A})}{\pushforward{A}\P} 
        \\
        &\subset 
        \ball{\varepsilon}{c \circ (\pinv{A} \times \pinv{A})}{\pushforward{A}\P},
    \end{aligned}
    \end{equation}
    with $\pinv{A} = (\adjoint{A} A)^{-1}\adjoint{A}$ and $\adjoint{A}=X^{-1} \T{A} Y$.
    
    \item Let let $c:\reals^n \to \nonnegativeReals$ satisfy \cref{assump:transportation:cost}. If $A \in \reals^{m \times n}$ has full row-rank, then
    \begin{align}
    \label{eq:B_epsilon:surjective:A:2}
        \pushforward{A}{\ball{\varepsilon}{c}{\P}} = \mathbb B_\varepsilon^{c \circ \pinv{A}}(\pushforward{A}\P),
    \end{align}
    with $\pinv{A} = \adjoint{A} (A \adjoint{A})^{-1}$ and $\adjoint{A}=X^{-1} \T{A} Y$.
\end{enumerate}
\end{corollary}
\end{comment}

We conclude this section with the application of \cref{thm:nonlin:trans,thm:lin:trans} to some important special cases.

\begin{corollary}
\label{cor:lin:trans:specialcase}
Let $\P\in\probSpace{\reals^n}$, $c:\reals^n\to\nonnegativeReals$ be a translation-invariant transportation cost over $\reals^n$, and $\varepsilon>0$.
\begin{enumerate}
    \item If $f:\reals^n\to\reals^n$ is a translation, i.e., $f(x) x + b$ for $b\in\reals^n$, then
    $%\begin{equation*}
        \pushforward{f}\ball{\varepsilon}{c}{\P}=\ball{\varepsilon}{c}{\pushforward{f}\P}.
    $%\end{equation*}
    \item If $f:\reals^n\to\reals^n$ is a scaling, i.e., $f(x)=\alpha x$ for $\alpha\in\reals$, and $c$ satisfies $c(\alpha (x_1-x_2))=|\alpha|^p c(x_1-x_2)$, for some $p\geq 0$, then
    $%\begin{equation*}
        \pushforward{f}\ball{\varepsilon}{c}{\P}=\ball{\varepsilon}{c/ \abs{\alpha}^p}{\pushforward{f}\P}=\ball{\abs{\alpha}^p \varepsilon}{c}{\pushforward{f}\P}.\;
    $%\end{equation*}
    In particular, if $\alpha=0$, then $\pushforward{f}\ball{\varepsilon}{c}{\P}=\{\diracDelta{0}\}$.
    \item If $f:\reals^n\to\reals^n$ is a rotation (resp., reflection) map, i.e., $f(x)=A x$, where $A$ is an orthogonal matrix with determinant 1 (resp., $-1$), and $c(x_1-x_2) = \psi(\norm{x_1-x_2})$ for $\psi:\nonnegativeReals\to\nonnegativeReals$, then
    $%\begin{equation*}
        \pushforward{A}\ball{\varepsilon}{c}{\P}=\ball{\varepsilon}{c}{\pushforward{A}\P}.
    $%\end{equation*}
    \item Let the assumptions of \cref{thm:lin:trans} be satisfied. If $A$ is an orthogonal projection matrix, i.e., $A$ satisfies $A = \adjoint{A}$ and $A^2=A$, then
    $%\begin{equation*}
        \pushforward{A}\ball{\varepsilon}{c}{\P} \subset  \ball{\varepsilon}{c\circ A}{\pushforward{A}\P}.  
    $%\end{equation*}
\end{enumerate}
\end{corollary}

% --------------------------------------------------------------------
% --------------------------------------------------------------------
% --------------------------------------------------------------------

\section{Convolution and Hadamard Product of \gls{acr:ot} Ambiguity Sets}
\label{sec:additive:multiplicative}

In this section, we study the effect of additive and multiplicative uncertainty. In the probability space, these two operations correspond to the convolution and Hadamard product of probability distributions (see~\cref{def:convolution,def:Hadamard}). As in the previous section, we do not restrict ourselves to fixed probability distributions but allow the uncertainty to be described by an \gls{acr:ot} ambiguity set. Accordingly we seek to characterize the two objects $\ball{\varepsilon_1}{c}{\P}\ast\ball{\varepsilon_2}{c}{\Q}$ and $\ball{\varepsilon_1}{c}{\P}\odot\ball{\varepsilon_2}{c}{\Q}$. Throughout this section, we assume that $\mathcal{X}=\reals^n$ (endowed with any distance).

% --------------------------------------------------------------------

\subsection{Convolution}
\label{subsec:additive:noise}

We start with the convolution of \gls{acr:ot} ambiguity sets. We proceed in three steps. In~\cref{lemma:coupling:convolution}, we prove a result at the level of the set of couplings. Armed with this result, in~\cref{prop:W:convolution} we study the convolution at the level of the \gls{acr:ot} discrepancy, which is used to define \gls{acr:ot} ambiguity sets. Our analysis culminates in \cref{thm:conv:trans}, where we show that the convolution of \gls{acr:ot} ambiguity sets can be captured by another \gls{acr:ot} ambiguity set. Throughout this subsection, we require the following assumption on the transportation cost.

\begin{assumption}
\label{assump:transportation:cost:conv}\,
\begin{enumerate}[(i)]
    \item
    The transportation cost $c$ is translation-invariant: $c(x_1,x_2)\coloneqq c(x_1 - x_2)$.

    \item
    There exists $p\geq 1$ so that $c^{\frac{1}{p}}$ satisfies triangle inequality; i.e., $c^{\frac{1}{p}}(x_1 - x_3) \leq c^{\frac{1}{p}}(x_1 - x_2) + c^{\frac{1}{p}}(x_2 - x_3)$ for all $x_1,x_2, x_3 \in \reals^n$.
\end{enumerate}
\end{assumption}

\cref{assump:transportation:cost:conv} is satisfied in particular by any power of a norm on $\reals^n$. At the level of coupling, we can provide a ``lower bound'': the set of couplings between $\P_1\ast\Q$ and $\P_2\ast\Q$ contains the set of couplings between $\P_1$ and $\P_2$ convolved with $\pushforward{(\tensorProd{\Id}{\Id})}\Q$. % and with $\Q\otimes\Q$.

\begin{lemma}
\label{lemma:coupling:convolution}
Let $\P_1,\P_2, \Q \in \probSpace{\reals^ n}$. Then,
%\begin{enumerate} 
%    \item $\setPlans{\P_1}{\P_2}\ast \pushforward{(\Id\times\Id)}\Q  \subset \setPlans{\P_1\ast\Q}{\P_2 \ast \Q}$;
\begin{equation}\label{eq:lemma:coupling:convolution}
    \setPlans{\P_1}{\P_2}\ast \pushforward{(\Id\times\Id)}\Q  \subset \setPlans{\P_1\ast\Q}{\P_2 \ast \Q}.
\end{equation}
%    \item $\setPlans{\P_1}{\P_2}\ast (\Q \otimes \Q)  \subset \setPlans{\P_1\ast\Q}{\P_2 \ast \Q}$.
%\end{enumerate}
\end{lemma}

\begin{remark}
In general, the inclusion~\eqref{eq:lemma:coupling:convolution} cannot be improved to equality. For instance, let $n=1$ and $\P_1=\P_2=\diracDelta{0}$. Then, $\setPlans{\P_1}{\P_2}=\{\diracDelta{(0,0)}\}$. Moreover, let $\Q= 0.5\,\diracDelta{1}+0.5\,\diracDelta{2}$. Then, $\setPlans{\P_1}{\P_2}\ast\pushforward{(\tensorProd{\Id}{\Id})}\Q=\{\pushforward{(\tensorProd{\Id}{\Id})}\Q\}=\{0.5\,\diracDelta{(1,1)}+0.5\,\diracDelta{(2,2)}\}$, i.e., it contains only one coupling. However, $\P_1\ast\Q=\P_2\ast\Q=\Q$, and so $\setPlans{\P_1\ast\Q}{\P_2\ast\Q}=\setPlans{\Q}{\Q}$, which is not a singleton since there are infinitely many couplings between $\Q$ and itself (e.g., $\pushforward{(\tensorProd{\Id}{\Id})}{\Q}$, $\Q\otimes\Q$, etc.). 
\end{remark}

We now study the \gls{acr:ot} discrepancy between two probability distributions, both convolved with $\Q$. Thanks to~\cref{lemma:coupling:convolution}, we can produce an upper bound on the \gls{acr:ot} discrepancy by restricting ourselves to couplings of the form $\setPlans{\P_1}{\P_2}\ast\pushforward{(\tensorProd{\Id}{\Id})}\Q$. This observation yields a contraction property.

\begin{proposition}
\label{prop:W:convolution}
Let $\P_1,\P_2, \Q \in \probSpace{\reals^ n}$ and let $c:\reals^n\to\nonnegativeReals$ satisfy~\cref{assump:transportation:cost:conv}(i). Then,
\begin{equation}\label{eq:prop:W:convolution}
    \transportCost{c}{\P_1 \ast \Q}{\P_2 \ast \Q}\leqslant \transportCost{c}{\P_1 }{\P_2}.
\end{equation}
\end{proposition}

\cref{prop:W:convolution} shows that the \gls{acr:ot} discrepancy between $\P_1\ast\Q$ and $\P_2\ast\Q$ is upper bounded by the \gls{acr:ot} discrepancy between $\P_1$ and $\P_2$.
Adding noise to probability distributions before computing the Wasserstein distance has been used in machine learning to devise a \emph{smoothed} version of the Wasserstein distance. 
%\cref{prop:W:convolution} generalizes the contraction property of the smoothed Wasserstein distance, proposed in~\cite{Goldfeld2022}, to general translation-invariant transportation costs.
%Since translation invariance of the transportation cost (i.e.,~\cref{assump:transportation:cost:conv}(i)) is equivalent to $c(x_1+x_3,x_2+x_3)\leq c(x_1,x_2)$ for all $x_1,x_2,x_3\in\reals^n$, \cref{prop:W:convolution} states the \gls{acr:ot} discrepancy $\transportCost{c}{\cdot}{\cdot}$ satisfies the analogous of translation invariance in the probability space, namely~\eqref{eq:prop:W:convolution}.
With~\cref{prop:W:convolution}, we can give a general result on the convolution of \gls{acr:ot} ambiguity sets. %Except for special cases (investigated below), this result takes the form of an upper bound for $\ball{\varepsilon_1}{c_1}{\P}\ast \ball{\varepsilon_2}{c_2}{\Q}$.

\begin{theorem}[Convolution]
\label{thm:conv:trans}
Let $\P, \Q \in \probSpace{\reals^ n}$, and let $c:\reals^n \to \nonnegativeReals$ satisfy \cref{assump:transportation:cost:conv}. Then,
\begin{equation}\label{thm:conv:trans:upperbound}
    \ball{\varepsilon_1}{c}{\P}\ast \ball{\varepsilon_2}{c}{\Q}
    \subset
    \ball{(\varepsilon_1^{1/p}+\varepsilon_2^{1/p})^p}{c}{\P\ast \Q}.
\end{equation}
\end{theorem}

\begin{remark}
In general, the upper bound~\labelcref{thm:conv:trans:upperbound} cannot be improved to equality. For instance, let $n=1$, $c(x_1-x_2)=|x_1-x_2|$, $\varepsilon_1=0.1$, and $\varepsilon_2 = 0$. Let $\P = \delta_{1}$ and $\Q=0.5\,\delta_{0.9}+0.5\,\delta_{1.1}$, and consider the \gls{acr:ot} ambiguity set $\ball{\varepsilon_1}{c}{\P \ast \Q}=\ball{0.1}{c}{0.5\,\delta_{1.9}+0.5\,\delta_{2.1}}$. The probability distribution $\diracDelta{2}$ belongs to $\ball{\varepsilon_1+\varepsilon_2}{c}{\P \ast \Q}$. However, $\diracDelta{2}\notin \ball{\varepsilon_1}{c}{\P} \ast \ball{\varepsilon_2}{c}{\Q}=\ball{\varepsilon_1}{c}{\P} \ast \Q$, since the convolution of any distribution in $\ball{\varepsilon_1}{c}{\P}$ with $\Q$ is supported on at least two distinct points. Thus, $\ball{\varepsilon_1}{c}{\P}\ast \Q\subsetneq \ball{\varepsilon_1}{c}{\P \ast \Q}$.
Nonetheless, in some special cases, we do have $\ball{\varepsilon_1}{c}{\P} \ast \Q = \ball{\varepsilon_1}{c}{\P\ast \Q}$; e.g., if $\Q = \diracDelta{x}$, then convolution with $\Q$ is a translation by $x$, and the equality follows from~\cref{cor:lin:trans:specialcase}.
\end{remark}

\cref{thm:conv:trans} continues to hold if the \gls{acr:ot} ambiguity sets $\ball{\varepsilon_1}{c}{\P}$ and $\ball{\varepsilon_2}{c}{\Q}$ are defined over some subsets $\mathcal X, \mathcal Y \subset \reals^n$. This is summarized in the next corollary, whose proof follows similar lines to the proof of \cref{thm:conv:trans} and is thus omitted. 

\begin{corollary}
\label{cor:conv:trans}
Let $\P \in \probSpace{\mathcal X}$ and $\Q \in \probSpace{\mathcal Y}$, with $\mathcal X, \mathcal Y \subset \reals^n$. Moreover, let $\ball{\varepsilon_1}{c}{\P}$ be defined over $\mathcal X$ and $\ball{\varepsilon_2}{c}{\Q}$ be defined over $\mathcal Y$. Finally, let $c:\reals^n \to \nonnegativeReals$ satisfy \cref{assump:transportation:cost:conv}. Then,
\begin{equation}\label{thm:conv:trans:upperbound:restricted}
    \ball{\varepsilon_1}{c}{\P}\ast \ball{\varepsilon_2}{c}{\Q}
    \subset
    \ball{(\varepsilon_1^{1/p}+\varepsilon_2^{1/p})^p}{c}{\P\ast \Q},
\end{equation}
with $\ball{(\varepsilon_1^{1/p}+\varepsilon_2^{1/p})^p}{c}{\P\ast \Q}$ restricted to all distributions supported on the Minkowski sum $\mathcal X \oplus \mathcal Y$.
\end{corollary}

We conclude this subsection by specializing \cref{thm:conv:trans} to the case of a known probability distribution $\Q$.

\begin{corollary}
\label{cor:conv:known:noise}
Let $\P, \Q \in \probSpace{\reals^ n}$, and let $c:\reals^n \to \nonnegativeReals$ satisfy \cref{assump:transportation:cost:conv}. If $\varepsilon_2=0$ and $c$ is positive definite, then $\ball{0}{c}{\Q} = \{\Q\}$ and $\ball{\varepsilon_1}{c}{\P}\ast \Q \subset \ball{\varepsilon_1}{c}{\P\ast \Q}$.
\end{corollary}

% --------------------------------------------------------------------

\subsection{Hadamard Product}
\label{subsec:multiplicative:noise}

We now study the Hadamard product of \gls{acr:ot} ambiguity sets. As above, we proceed in three steps. In~\cref{lemma:coupling:Hadamard}, we prove a result at the level of the set of couplings. Armed with this result, in~\cref{prop:W:Hadamard} we study the Hadamard product at the level of the \gls{acr:ot} discrepancy. Our analysis culminates in \cref{thm:Hadamard}, where we show that the Hadamard product of \gls{acr:ot} ambiguity sets can be captured by another \gls{acr:ot} ambiguity set. Throughout this subsection, we require the following assumption on the transportation cost.

\begin{assumption}
\label{assump:transportation:cost:Hadamard}
\,
\begin{enumerate}[(i)]
    \item The transportation costs $c$ is translation-invariant and satisfies $c(x_1\cdot x_3-x_2\cdot x_3)\leq c(x_1-x_2)\, c(x_3)$ for all $x_1, x_2, x_3 \in \reals^n$.

    \item There exists $p\geq 1$ such that $c^{\frac{1}{p}}$ satisfies triangle inequality; i.e., $c^\frac{1}{p}(x_1 - x_3) \leq c^\frac{1}{p}(x_1 - x_2) + c^\frac{1}{p}(x_2 - x_3)$ for all $x_1,x_2, x_3 \in \reals^n$;
    
\end{enumerate}

\end{assumption}

\cref{assump:transportation:cost:Hadamard} is satisfied in particular by any power of a norm on $\reals^n$. Similarly to~\cref{lemma:coupling:convolution} for the convolution operation, at the level of coupling we can provide a ``lower bound'': the set of couplings between $\P_1\odot\Q$ and $\P_2\odot\Q$ contains Hadamard product of $\setPlans{\P_1}{\P_2}$ and $\pushforward{(\tensorProd{\Id}{\Id})}\Q$.

\begin{lemma}
\label{lemma:coupling:Hadamard}
Let $\P_1,\P_2, \Q \in \probSpace{\reals^n}$. Then,
\begin{align*}
    \setPlans{\P_1}{\P_2}\odot \pushforward{(\Id\times\Id)}\Q  \subset \setPlans{\P_1\odot\Q}{\P_2 \odot \Q}.
\end{align*}
\end{lemma}

We now study the \gls{acr:ot} discrepancy between two $\P_1 \odot \Q$ and $\P_2 \odot \Q$. Thanks to~\cref{lemma:coupling:Hadamard}, we can produce an upper bound of the \gls{acr:ot} discrepancy by restricting ourselves to couplings of the form $\setPlans{\P_1}{\P_2}\odot\pushforward{(\tensorProd{\Id}{\Id})}\Q$. 

\begin{proposition}
\label{prop:W:Hadamard}
Let $\P_1,\P_2, \Q \in \probSpace{\reals^n}$, and let $c:\reals^n \to \nonnegativeReals$ satisfy~\cref{assump:transportation:cost:Hadamard}(i). Then, 
\begin{equation}\label{eq:prop:W:Hadamard:inequality}
    \transportCost{c}{\P_1 \odot \Q}{\P_2 \odot \Q}\leqslant\transportCost{c}{\P_1 }{\P_2}\, \mathbb{E}_{\Q}[c(x)].
\end{equation}
\end{proposition}

\cref{prop:W:Hadamard} states that the \gls{acr:ot} discrepancy satisfies the analogous of~\cref{assump:transportation:cost:Hadamard}(i) in the probability space, namely the inequality~\eqref{eq:prop:W:Hadamard:inequality}.
Armed with~\cref{prop:W:Hadamard}, we can now study the Hadamard product of \gls{acr:ot} ambiguity sets.

\begin{theorem}[Hadamard product]
\label{thm:Hadamard}
Let $\P, \Q \in \probSpace{\reals^n}$, and let $c:\reals^n \to \nonnegativeReals$ satisfy~\cref{assump:transportation:cost:Hadamard}. Then,
\begin{equation}\label{eq:thm:hadamard}
\begin{aligned}
    \ball{\varepsilon_1}{c}{\P}&\odot \ball{\varepsilon_2}{c}{\Q}
    \\
    &\subset  \ball{(\varepsilon_1^{1/p}\varepsilon_2^{1/p} +\varepsilon_1^{1/p}  {\mathbb{E}_{\Q}[c(x)]}^{1/p} + \varepsilon_2^{1/p} {\mathbb{E}_{\P}[c(x)]}^{1/p})^p}{c}{\P \odot \Q}.
\end{aligned}
\end{equation}
\end{theorem}

\begin{remark}
In general, the upper bound~\labelcref{eq:thm:hadamard} cannot be improved to equality. As for the convolution, let $n=1$, $c(x_1-x_2)=|x_1-x_2|$, $\varepsilon_1=0.1$, and $\varepsilon_2=0$. Let $\P = \delta_{1}$ and $\Q=0.5\,\delta_{0.9}+0.5\,\delta_{1.1}$, and consider the \gls{acr:ot} ambiguity set $\ball{\varepsilon_1}{c}{\P \odot \Q}=\ball{0.1}{c}{0.5\,\delta_{0.9}+0.5\,\delta_{1.1}}$. The probability distribution $\diracDelta{1}$ belongs to $\ball{\varepsilon_1}{c}{\P \odot \Q}$. However, $\diracDelta{1}\notin \ball{\varepsilon_1}{c}{\P} \odot \Q$, since the Hadamard product of any distribution in $\ball{\varepsilon_1}{c}{\P}$ with $\Q$ is supported on at least two distinct points (with the only possible exception for $\diracDelta{0}$). Thus, $\ball{\varepsilon_1}{c}{\P}\odot \Q\subsetneq \ball{\varepsilon_1}{c}{\P \odot \Q}$.
%Nonetheless, in some special cases, we do have $\ball{\varepsilon_1}{c}{\P} \ast \Q = \ball{\varepsilon_1}{c}{\P\ast \Q}$. For example, if $\Q = \diracDelta{x}$, then convolution with $\Q$ reduces to a translation by $x$, and the equality of sets follows from~\cref{cor:lin:trans:specialcase}.
\end{remark}

We conclude this subsection by specializing \cref{thm:Hadamard} to the case of a known probability distribution $\Q$.

\begin{corollary}
\label{cor:hadamard:known:noise}
Let $\P, \Q \in \probSpace{\reals^ n}$, and let $c:\reals^n \times \reals^n \to \nonnegativeReals$ satisfy~\cref{assump:transportation:cost:Hadamard}. If $\varepsilon_2=0$ and $c$ is positive definite, then $\ball{\varepsilon_1}{c}{\P}\odot \Q \subset \ball{\varepsilon_1  \mathbb{E}_{\Q}[c(x)]}{c}{\P\odot \Q}$.
\end{corollary}

% --------------------------------------------------------------------
% --------------------------------------------------------------------
% --------------------------------------------------------------------

\section{Applications}
\label{sec:applications}

In this section, we specialize the results presented in~\cref{thm:lin:trans,thm:nonlin:trans,thm:conv:trans,thm:Hadamard} to several applications in the context of (stochastic) dynamical systems and least squares estimation. 

%In this section, we present several applications.
%In~\cref{sec:app:dynamical:linear}, we leverage our tools to propagate distributional uncertainty in stochastic \gls{acr:lti} control systems. In~\cref{sec:application:trajectory:planning,sec:application:consensus}, we study distributionally robust trajectory planning and uncertainty quantification of the consensus state of discrete-time averaging algorithms. 
%In~\cref{sec:LS}, we capture the distributional uncertainty in the error for the least squares estimator. Finally, in~\cref{sec:app:dynamical:nonlinear}, we present preliminary results for nonlinear dynamical systems and discuss related future research directions. 

\subsection{Stochastic Linear Control Systems}\label{sec:app:dynamical:linear}
Our first application concerns the propagation of distributional uncertainty through a stochastic \gls{acr:lti} control system.
We focus on three sources of uncertainty: stochastic initial condition, additive process noise, and multiplicative noise.
In all cases, we capture distributional uncertainty via \gls{acr:ot} ambiguity sets (using the translation-invariant and orthomonotone transportation cost $c(x,y)=\norm{x-y}_2^2$) and aim at quantifying the distributional uncertainty in the state $x_t\in\reals^n$ at time $t$. 

Specifically, our three settings of interest are as follows: 
\begin{enumerate}
    \item Deterministic \gls{acr:lti} system with uncertain initial condition, whose distribution belongs to an \gls{acr:ot} ambiguity set:
    \begin{align}
    \label{eq:app:linear:random:initial}
        x_{t+1} = Ax_t+Bu_t, \quad x_0 \sim \Q_0 \in \ball{\varepsilon}{\norm{\cdot}_2^2}{\P_0}.
    \end{align} 

    \item Stochastic \gls{acr:lti} system with deterministic initial condition, and with additive noise, where the distribution of the noise trajectory belongs to an \gls{acr:ot} ambiguity set:
    \begin{multline}
        \label{eq:app:linear:additive:noise}
        x_{t+1} = Ax_t+Bu_t+Dw_t, \quad \\\mathbf{w}_{[t-1]}=\begin{bmatrix} w_{t-1}^\top &\cdots & w_0^\top\end{bmatrix}^\top \sim \Q_t \in \ball{\varepsilon}{\norm{\cdot}_2^2}{\P_t}.
    \end{multline} 

    \item Stochastic \gls{acr:lti} system with deterministic initial condition, and with two sources of independent (element-wise) multiplicative noise whose distributions belong to two \gls{acr:ot} ambiguity sets: 
    \begin{multline}
    \label{eq:app:linear:multiplicative:noise}
        x_{t+1} = w^{(1)}\cdot Ax_t+ w^{(2)} \cdot Bu_t,
        \\
        w_1 \sim \Q^{(1)} \in \ball{\varepsilon_1}{\norm{\cdot}_2^2}{\P^{(1)}}, 
        w_2 \sim \Q^{(2)} \in \ball{\varepsilon_2}{\norm{\cdot}_2^2}{\P^{(2)}}.
    \end{multline}
\end{enumerate}

These formulations encompass the scenario when one only has access to $N$ i.i.d. samples of the uncertainty and constructs an \gls{acr:ot} ambiguity set around these.
For instance, in the case of unknown initial condition (case 1) above), suppose one has access to $N$ samples $\{\sample{x}{i}_0\}_{i=1}^N$ from $\Q_0$. In this case, statistical concentration inequalities can be used to construct an \gls{acr:ot} ambiguity set $\ball{\varepsilon}{c}{\widehat\Q_{0}}$ around the empirical distribution $\widehat\Q_{0} = \frac{1}{N}\sum_{i=1}^N \diracDelta{\sample{x}{i}_0}$ which contains the true distribution $\P$ with high probability; see~\cite[Section~4]{aolaritei2023capture} for details. 

We now use the theory of~\cref{sec:propagation,sec:additive:multiplicative} to capture the distributional uncertainty in $x_t$. To ease notation, let
\begin{align}
\label{eq:stoch:dyn:sys:0-t}
\begin{split}
    \mathbf{u}_{[t-1]} &\coloneqq
    \begin{bmatrix}u_{t-1}^\top & u_{t-2}^\top & \cdots & u_0^\top\end{bmatrix}^\top,
    \\
    \mathbf{B}_{t-1} &\coloneqq \begin{bmatrix}B & A B & \cdots & A^{t-1} B \end{bmatrix},
    \\ 
    \mathbf{D}_{t-1} &\coloneqq \begin{bmatrix} D & A D & \cdots & A^{t-1} D \end{bmatrix}.
\end{split}
\end{align}

% where $u_t\in\reals^m$ and $A\in\reals^{n\times n}, B\in\reals^{n\times m}$.

\begin{proposition}[Distributional uncertainty in \gls{acr:lti} systems]\label{prop:applications:propagation:linear}
\,
\begin{enumerate}
    \item Consider the system~\eqref{eq:app:linear:random:initial}. Then, 
    \begin{equation*}
        x_t\sim \Q_t\in
        \ball{\varepsilon}{\norm{\cdot}_2^2\circ \pinv{(A^t)}}{\delta_{\mathbf{B}_{t-1}\mathbf{u}_{[t-1]}} \ast(\pushforward{A^t}\P_0)}.
    \end{equation*}

    \item Consider the system~\eqref{eq:app:linear:additive:noise}. Then,
    \begin{equation*}
        \hspace{-0.5cm}
        x_t\sim \Q_t\in
        \ball{\varepsilon}{\|\cdot\|_2^2\circ \pinv{\mathbf{D}_{t-1}}}{\delta_{A^tx_0 + \mathbf{B}_{t-1}\mathbf{u}_{[t-1]}} \ast (\pushforward{\mathbf{D}_{t-1}}\P_t)}.
    \end{equation*}

    \item Consider the system~\eqref{eq:app:linear:multiplicative:noise}. Then, 
    \begin{equation*}
        x_t\sim \Q_t\in
        \ball{\rho_t}{\|\cdot\|_2^2}{\P_t}
    \end{equation*}
    with $\P_t\in\Pp{}{\reals^n}$ and $\rho_t\geq 0$ defined recursively by $\P_0=\diracDelta{x_0}$, $\rho_t=0$, and
    \begin{align*}
        \P_t
        &=
        \left(\P^{(1)}\odot \pushforward{A}\P_{t-1}\right)\ast \left(\P^{(2)}\odot \diracDelta{Bu_{t-1}}\right)
        \\
        \rho_t
        &=
        \begin{aligned}[t]
        \big(\sqrt{\varepsilon_1 \rho_{t-1}}&\sigma_\mathrm{\max}(A)+\sqrt{\rho_{t-1}M_{\P^{(1)}}}\sigma_\mathrm{\max}(A) \\
        &+\sqrt{\varepsilon_{1}M_{\pushforward{A}\P_{t-1}}}+\sqrt{\varepsilon_2}\norm{Bu_{t-1}}_2\big)^2,
        \end{aligned}
    \end{align*}
    where $\sigma_\mathrm{\max}(A)$ is the maximum singular value of $A$ and $M_\P\coloneqq\int_{\reals^n}\norm{x}^2\d\P(x)$ is the second moment of $\P$.
\end{enumerate}
\end{proposition}

\begin{remark}
The propagation in 1) is exact if $A$ is full row-rank and the propagation in 2) is exact if $\mathbf{D}_{t-1}$ is full row-rank. In all other cases,~\cref{prop:applications:propagation:linear} provides us with a (non-trivial) upper bound for the distributional uncertainty in $x_t$. 
\end{remark}

\begin{remark}
The results in \cref{thm:lin:trans,thm:conv:trans,thm:Hadamard} can also be employed in scenarios where multiple (simultaneous) sources of uncertainty are present. For example, for stochastic~\gls{acr:lti} systems with both uncertain initial condition and additive noise, i.e., 
\begin{align*}
    x_{t+1} &= Ax_t+Bu_t+Dw_t, \quad x_0 \sim \Q_0 \in \ball{\varepsilon_1}{\norm{\cdot}_2^2}{\P_0}, \\ \mathbf{w}_{[t-1]} &=\begin{bmatrix} w_{t-1}^\top &\cdots & w_0^\top\end{bmatrix}^\top \sim \Q_t \in \ball{\varepsilon_2}{\norm{\cdot}_2^2}{\P_t}
\end{align*}
it can be easily shown that the OT ambiguity set
\begin{equation*}
    \ball{\tilde{\varepsilon}}{\|\cdot\|_2^2}{\delta_{\mathbf{B}_{t-1}\mathbf{u}_{[t-1]}} \ast(\pushforward{A^t}\P_0)\ast (\pushforward{\mathbf{D}_{t-1}}\P_t)},
\end{equation*}
with $\tilde{\varepsilon}\coloneqq(\sqrt{\varepsilon_1}/\sigma_{\text{max}}(A^t) + \sqrt{\varepsilon_2}/\sigma_{\text{max}}(\mathbf{D}_{t-1}))^2$, captures the distributional uncertainty in the state $x_t$. Notice that this ambiguity set is not the result of a simple superposition of the two ambiguity sets in 1)-2) from \cref{prop:applications:propagation:linear}.
\end{remark}

\subsection{Distributionally Robust Trajectory Planning}
\label{sec:application:trajectory:planning}

\begin{figure}[t]
    \centering
    \includegraphics[width=0.9\linewidth]{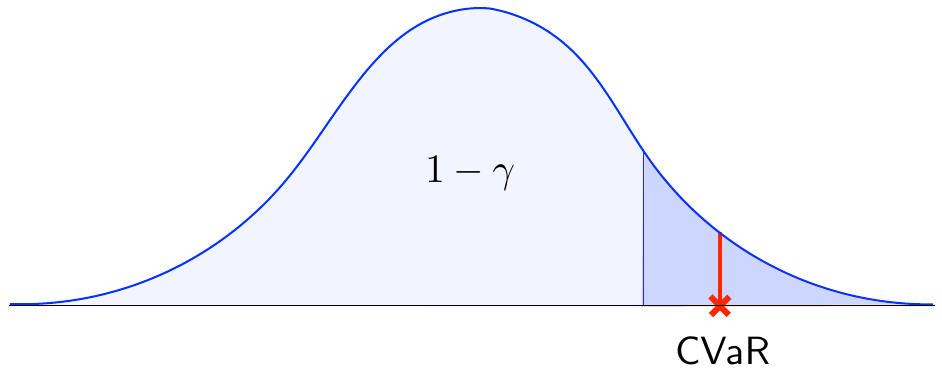}
    \caption{CVaR is the expected value of the right $\gamma$-tail (dark blue).}
    \label{fig:cvar}
\end{figure}

We now instantiate our results for distributionally robust trajectory planning. Consider the stochastic \gls{acr:lti} system
\begin{equation}\label{eq:app:trajectory:lti}
    x_{t+1}=Ax_t+Bu_t+Dw_t,
\end{equation}
where $x_t\in\reals^n, u_t\in\reals^m$, $w_t\in\reals^r$, and all system matrices are of appropriate dimensions. We suppose that the noise trajectory $\mathbf{w}_{[t-1]}=\begin{bmatrix} w_{t-1}^\top & \cdots & w_0^\top\end{bmatrix}^\top$ is distributed according to an \emph{unknown} probability distribution $\P_t\in\Pp{}{\reals^{rt}}$. In particular, we do not assume that $w_i$ and $w_j$ are uncorrelated. 

For a horizon $T\in\naturals$, we aim at steering \eqref{eq:app:trajectory:lti} from a given initial condition $x_0\in\reals^n$ to a target set $\mathbb{X}\subset\reals^n$ defined as
\begin{equation}\label{eq:app:polyhedron}
    \mathbb X \coloneqq %\{x \in \reals^n:\, A x \leq b\}\\
    \left\{x \in \reals^n:\, \max_{j \in [J]} a_j^\top x + b_j \leq 0,\; J \in \mathbb N\right\},
\end{equation}
where $a_j\in\reals^n$ and $b_j\in\reals$. In what follows, we assume that $\mathbf{D}_{T-1}$ defined in \eqref{eq:stoch:dyn:sys:0-t} is full row-rank.
To deal with distributions with unbounded support and reduce conservatism,
we say an input trajectory $\mathbf{u}_{[T-1]}$ is \emph{feasible} if the probability distribution of the terminal state $x_T$, given by 
\begin{equation}\label{eq:app:dr:steering:exact}
    \Q_T(\mathbf{u}_{[T-1]})=\delta_{A^T x_0 + \mathbf{B}_{T-1} \mathbf{u}_{[T-1]}}\ast(\pushforward{\mathbf{D}_{T-1}}\P_T),
\end{equation}
satisfies the~\gls{acr:cvar} constraint:
\begin{equation}\label{eq:app:CVaR:constraint}
    \cvar_{1-\gamma}^{\Q_T(\mathbf{u}_{[T-1]})}
    \left(\max_{j \in [J]} a_j^\top x_T + b_j\right) \leq 0.
\end{equation} 
Due to space constraints, we refer to \cite[Equation~(1)]{aolaritei2023capture} for the specific mathematical formulation of ~\gls{acr:cvar} constraints, and simply recall here that: (i) $\cvar_{1-\gamma}^{\Q_T(\mathbf{u}_{[T-1]})}$ represents the expectation of the right $\gamma$-tail of the distribution $\Q_T(\mathbf{u}_{[T-1]})$ (see~\cref{fig:cvar}), and (ii) \eqref{eq:app:CVaR:constraint} is a convex constraint which implies that the terminal state $x_T$ belongs to the target set $\mathbb{X}$ with probability of at least $1-\gamma$.

If the true distribution of the noise trajectory $(w_0,\ldots,w_{T-1})$ were known, we could directly leverage~\eqref{eq:app:dr:steering:exact} and compute the cheapest control input via
\begin{equation}\label{eq:dr trajectory}
\begin{array}{cl}
        \min & \displaystyle\norm{\mathbf{u}_{[T-1]}}_2^2 \\
        \st
        & \DS \cvar_{1-\gamma}^{\Q_T(\mathbf{u}_{[T-1]})}\left(\max_{j \in [J]} a_j^\top x_T + b_j\right) \leq 0.    
\end{array}
\end{equation}
We consider instead the case where the distribution is unknown and we only have access to $N$ i.i.d. noise sample trajectories 
\begin{align*}
    \widehat{\mathbf{w}}_{[T-1]}^{(i)}
    \coloneqq 
    \begin{bmatrix}(\widehat{w}_{T-1}^{(i)})^{\top} & \cdots & (\widehat{w}_0^{(i)})^{\top} \end{bmatrix}^\top,
\end{align*}
for $i \in \{1,\ldots,N\}$. We capture the distributional uncertainty in the noise trajectories via an~\gls{acr:ot} ambiguity set $\ball{\varepsilon}{\norm{\cdot}_2^2}{\widehat\P_{T}}$ centered at the empirical distribution $\widehat\P_{T}=\frac{1}{N}\sum_{i=1}^N\delta_{\hat{\mathbf{w}}_{[T-1]}^{(i)}}$.
By~\cref{prop:applications:propagation:linear}, the distributional uncertainty in the terminal state $x_T$ is therefore captured by the \gls{acr:ot} ambiguity set
\begin{equation*}
\begin{aligned}
    \mathbb{S}(\mathbf{u}_{[T-1]})
    \coloneqq
    \ball{\varepsilon}{\|\cdot\|_2^2\circ \pinv{\mathbf{D}_{T-1}}}{\delta_{A^T x_0 + \mathbf{B}_{T-1} \mathbf{u}_{[T-1]}}\ast(\pushforward{\mathbf{D}_{T-1}}{\widehat\P_{T}})}.
\end{aligned}
\end{equation*}
In particular, the center of the \gls*{acr:ot} ambiguity set $\mathbb{S}(\mathbf{u}_{[T-1]})$ is supported on the $N$ controlled state samples
\begin{align*}
    \widehat{x}_T^{(i)}(\mathbf{u}_{[T-1]})
    = A^Tx_0 + \mathbf{B}_{T-1} \mathbf{u}_{[T-1]} + \mathbf{D}_{T-1} \widehat{\mathbf{w}}_{[T-1]}^{(i)}.
\end{align*}
Accordingly, the optimal input must satisfy the~\gls{acr:cvar} constraint~\eqref{eq:app:CVaR:constraint} for all probability distribution belonging to the \gls{acr:ot} ambiguity set $\mathbb{S}(\mathbf{u}_{[t-1]})$:
\begin{equation}\label{eq:dr trajectory:optimization}
\begin{array}{cl}
    \min & \displaystyle\norm{\mathbf{u}_{[T-1]}}_2^2 \\
    \st & \DS \sup_{\Q_T \in \mathbb{S}(\mathbf{u}_{[T-1]})} \cvar_{1-\gamma}^{\Q_T}\left(\max_{j \in [J]} a_j^\top x_T + b_j\right) \leq 0.
\end{array}
\end{equation}
The semi-infinite problem \eqref{eq:dr trajectory:optimization} can be reformulated as a finite-dimensional convex program using \cite[Proposition~2.12 and Theorem~2.16]{shafieezadeh2023new}. Due to space reasons, its proof is omitted.
\begin{proposition}[Distributionally robust trajectory planning]
\label{prop:app:distributionally robust trajectory:reformulation}
The distributionally robust trajectory planning~\eqref{eq:dr trajectory:optimization} admits the finite-dimensional convex reformulation
\begin{align*}
    \begin{array}{cll}
        \min & \DS \|\mathbf{u}_{[T-1]}\|_2^2 \\
        [1.5ex]\st 
        & \DS 
        \tau \in \reals, \lambda \in \reals_+, s_i \in \reals, &\forall i\in[N]
        \\
        & \DS \lambda \varepsilon N + \sum_{i=1}^{N} s_{i} \leq 0
        \\
        & \DS \frac{1}{4\lambda\gamma} \tilde{\alpha}_j + a_j^\top \widehat{x}_T^{(i)}(\mathbf{u}_{[T-1]}) + \tilde{\beta}_j \leq \gamma s_i
        &\forall i \in [N], j\in [J]
        \\ [1.5ex]
        &  \DS
         \tau\leq s_i  &\forall i\in[N].
    \end{array}
\end{align*}
with $\tilde{\alpha}_j := a_j^\top \left((\pinv{\mathbf{D}_{T-1}})^\top \pinv{\mathbf{D}_{T-1}} \right)^{-1} a_j$, and $\tilde{\beta}_j := b_j + \gamma \tau - \tau$.
\end{proposition}
%The proof of \cref{prop:app:distributionally robust trajectory:reformulation} follows readily from the definition of CVaR, and \cite[Proposition~2.12]{shafieezadeh2023new}. 
We evaluate our methodology on the two-dimensional linear system $A=\frac{1}{2}\begin{bsmallmatrix} 1 & -1 \\ 2 & 1\end{bsmallmatrix}$, $B=I$, and $D=0.1 I$, prestabilized with the LQR control gain (designed with $Q=R=I$) and $T=10$.
We suppose that the decision-maker has access to 5 noise sample trajectories. We choose the set $[1,2]\times [1,2]$ as the target (grey in~\cref{fig:trajectory:planning}) and select $\gamma=0.1$. We repeat our experiments for three values of $\varepsilon$ and compare with propagation of the \gls{acr:ot} ambiguity set via the Lipschitz constant and propagation of the center only.
As shown in \cref{fig:trajectory:planning}, the feedforward input resulting from $\varepsilon=0$ (red in~\cref{fig:trajectory:planning}) performs well on the 5 training sample trajectories, steering them to the boundary of the target set, but yields poor performance on unseen test samples. For larger $\varepsilon$ (blue and green in~\cref{fig:trajectory:planning}), instead, the system trajectories are successfully steered to the target set, even for unseen noise realizations, at the price of a moderate increase in cost (8.5\% larger cost for $\varepsilon=0.1$ and 14.9\% larger cost for $\varepsilon=0.3$, compared with $\varepsilon=0$). Finally, the propagation via the Lipschitz constant yields instead 29.3\% larger cost, whereas only propagating the center leads to infeasibility and is therefore not included in the plot. 

\begin{figure}[t]
    %\centering
    \begin{subfigure}[t]{0.585\columnwidth}
    \centering
    \includegraphics[height=6cm,trim={2.6cm 1cm 4cm 2.8cm},clip]{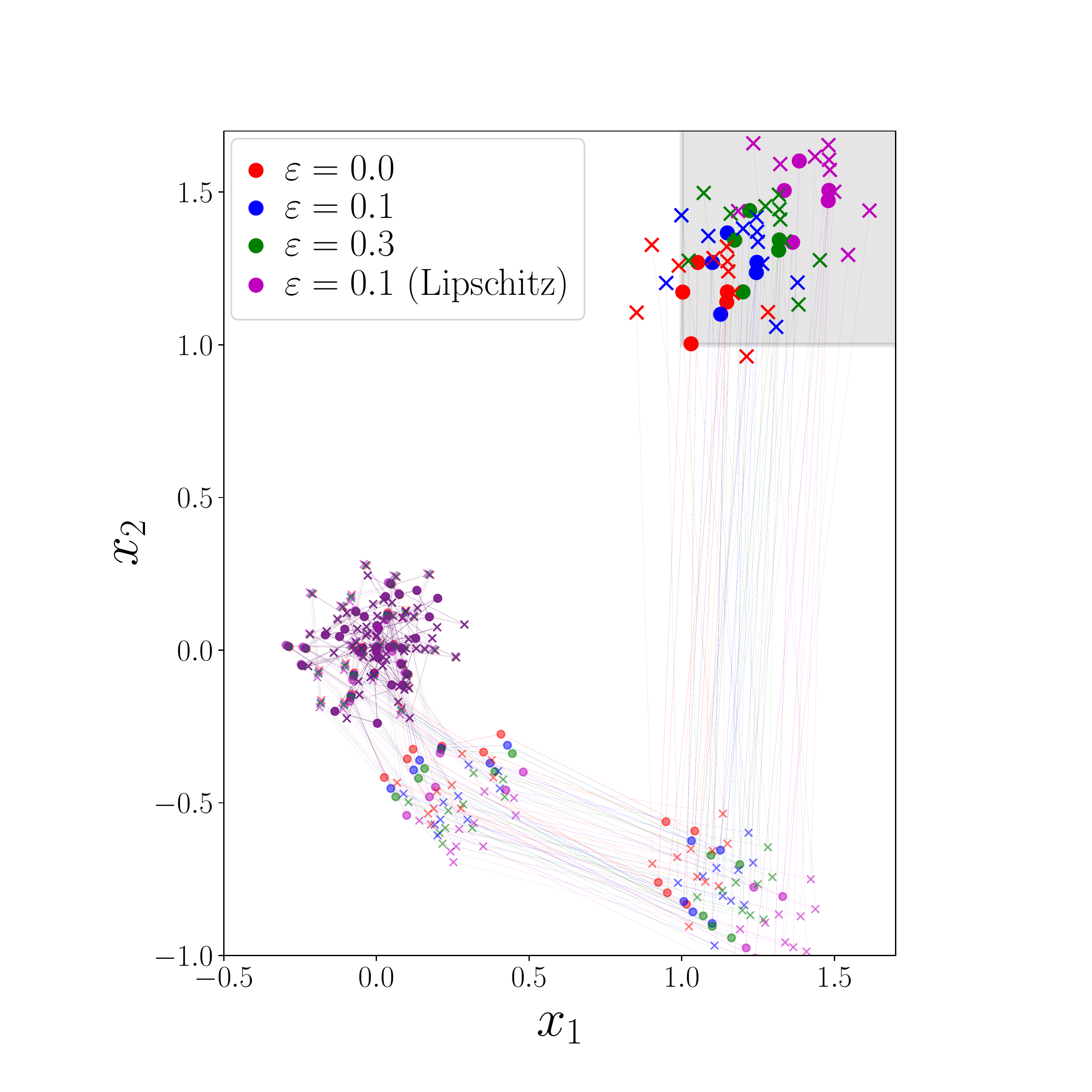}
    \caption{Training and testing trajectories for various radii; see plot on the right for zoom on the target set.}
    \end{subfigure}
    ~
    \begin{subfigure}[t]{0.385\columnwidth}
    \centering
    \includegraphics[height=6cm,trim={6.5cm 1cm 7.5cm 2.8cm},clip]{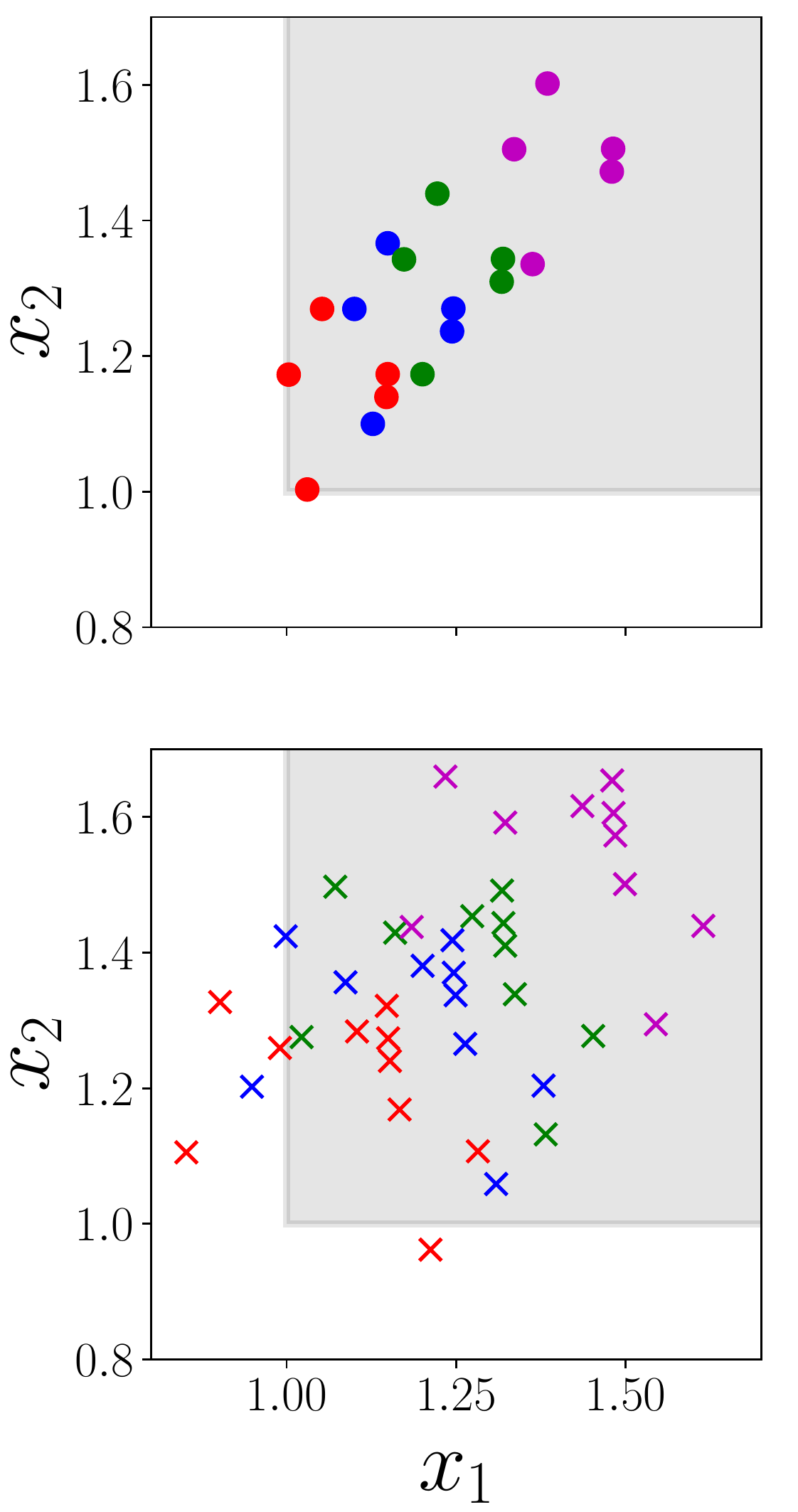}
    \caption{Training samples (top) and testing samples (bottom) for various radii.}
    \end{subfigure}
    \caption{Distributionally robust steering of a system from the origin to a target set for various radii $\varepsilon$. Filled circles (with solid lines) are training samples and crosses (with dotted lines) are testing samples.% Small radii lead to trajectories approaching the boundary of the target, which however might not reach the target under different noise realizations (in red); larger radii, instead, push the trajectory to the interior of the target (in green). We compare the exact propagation with the propagation using the Lipschitz constant, which leads to overly conservative solutions. The propagation of the center only, instead, leads to infeasibility.
    }
    \label{fig:trajectory:planning}
\end{figure}

\subsection{Consensus in Averaging Algorithms}
\label{sec:application:consensus}
Our third application concerns uncertainty quantification in the context of discrete-time averaging algorithms. More specifically, we study how distributional uncertainty in the initial condition of averaging algorithms propagates to the consensus state that (under suitable assumption) the algorithm reaches. 
Formally, consider the discrete-time averaging algorithm
\begin{equation}\label{eq:application:consensus:system}
    x_{t+1}=Ax_t,\qquad x_0\sim\Q_0\in\probSpace{\reals^n},
\end{equation}
where $A\in\reals^{n\times n}$ is the \emph{row-stochastic} adjacency matrix associated with a directed graph $G$ (which captures the interactions among the nodes/agents of the averaging). Under suitable assumptions on $G$, the iterate resulting from~\eqref{eq:application:consensus:system} is known to reach consensus (i.e., all entries of the vector $\lim_{t\to\infty}x_t$ are equal). 
Such algorithms are ubiquitous in many applications, including network systems, multi-agent systems, and opinion dynamics; see~\cite{bullo2020lectures} and references therein. 

In many real-world settings, the initial probability distribution $\Q_0$ is unknown. 
For instance, in opinion dynamics, the distribution of users' initial opinions is unknown but can be estimated via surveys (from which $\P_0$ is constructed); in sensor networks, $\P_0$ can be calibrated based on the sensors' specifications. 
Accordingly, we can use an \gls*{acr:ot} ambiguity set $\ball{\varepsilon}{\norm{\cdot}_2^2}{\P_0}$ to capture its distributional uncertainty. 

In this setting, we are interested in capturing the distributional uncertainty in the consensus state $\lim_{t\to\infty}x_t$.

\begin{proposition}[Consensus \gls{acr:ot} ambiguity set]
\label{prop:app:consensus}
Let $\Q_0 \in \ball{\varepsilon}{\norm{\cdot}_2^2}{\P_0}$, and let $A$ be row stochastic with a unique eigenvalue in $1$ and all the other eigenvalues strictly inside the unit circle. Moreover, let $\Q_t$ be the distribution of the state $x_t$, and $\Q_\infty$ be the distribution of $x_\infty := \lim_{t\to\infty}x_t$. Finally, let $w$ denote the left eigenvector of $A$ associated to the eigenvalue $1$ (normalized so that $\T \ones_n w =1$). Then,
\begin{align}\label{eq:app:consensus:one}
    \Q_t \to \Q_\infty \in \pushforward{(\ones_n)} \ball{\varepsilon\norm{w}_2^2}{\abs{\cdot}^2}{\pushforward{\T{w}} \P_0},
    %\subset \otimes_{i=1}^d \mathbb B_{|\entry{v}{i}|^p\varepsilon}^{c \circ \pinv{w}}(\pushforward{w} \P)
\end{align}
where $\rightarrow$ denotes the standard weak convergence. If additionally $A$ is doubly-stochastic, then $w=\frac{1}{n}\ones_n$ and
\begin{align}\label{eq:app:consensus:two}
    \Q_t \to \Q_\infty \in \pushforward{(\ones_n)} \ball{\varepsilon/n}{\abs{\cdot}^2}{\pushforward{\textstyle{\frac{1}{n}}\T{\ones_n}} \P_0}.
    %\subset \otimes_{i=1}^d \mathbb B_{|\entry{v}{i}|^p\varepsilon}^{c \circ \pinv{w}}(\pushforward{w} \P)
\end{align}
\end{proposition}

In words, the distributional uncertainty is the same in all entries of $\lim_{t\to\infty} x_t$, and it belongs to the \gls*{acr:ot} ambiguity set $\ball{\varepsilon\norm{w}_2^2}{\abs{\cdot}^2}{\pushforward{\T{w}} \P_0}$, which therefore captures the distributional uncertainty in the consensus state. This way, for instance, we can quantify the worst-case probability that the consensus state resides outside of an interval $(a,b)$ of interest via
\begin{equation*}
    \sup_{\P \in \ball{\varepsilon\norm{w}^2_2}{\abs{\cdot}^2}{\pushforward{\T{w}} \P_0}} \left\{\mathrm{Pr}(x\notin (a,b))=\int_{\reals\setminus(a,b)} \d\P(x)\right\}.
\end{equation*}
As above, this optimization problem admits a finite-dimensional convex reformulation \cite{shafieezadeh2023new}.

\subsection{Least Squares Estimation}\label{sec:LS}

Our fourth application concerns statistical estimation.
Consider the linear inverse problem
\begin{align}
\label{eq:least:squares}
    y = A x_0 + z,
\end{align}
where the aim is to estimate the parameter $x_0 \in \reals^n$ from $m$ noisy measurements (with $m>n$) of the form $\entry{y}{i} = \T{a_i} x_0 + z_i$, with $z_i\in\reals$ random measurement noise (possibly correlated). If $A$ has full column-rank, a standard solution is the \gls{acr:ols} estimator
\begin{align}
\label{eq:least:squares:estimator}
    \hat{x}\coloneqq \argmin_{x \in \reals^n}\; \|y-Ax\|_2=\pinv{A} y = (\T A A)^{-1}\T A y.
\end{align}
The \gls{acr:ols} estimator corresponds to the maximum likelihood estimation when the measurement noise $\{z_i\}_{i=1}^m$ are i.i.d. and standard normal with equal variances. Moreover, by the Gauss–Markov theorem \cite{sullivan2015introduction}, if $\{z_i\}_{i=1}^m$ are i.i.d. and have zero mean and equal finite variance, then the \gls{acr:ols} estimator is the best linear unbiased estimator, i.e., it has the lowest sampling variance within the class of linear unbiased estimators. 

Here, we depart from these assumptions and study a more general case, where we only have access to the $N$ i.i.d. samples $\{\sample{z}{i}\}_{i=1}^N$, with $\sample{z}{i}\in\reals$, from an otherwise unknown measurement noise (joint) distribution $\P\in\Pp{}{\reals^m}$ (see remark below for the specialization to the i.i.d. case). In this scenario, we seek to understand how the distributional uncertainty in measurement noise, captured by an \gls{acr:ot} ambiguity set, propagates to the \gls{acr:ols} estimator $\hat{x}$ in~\eqref{eq:least:squares:estimator}. Especially when the \gls{acr:ols} estimator is subsequently used for decision-making (e.g., in estimation or identification problems), it is crucial to accurately quantify uncertainty~\cite{summers2021distributionally}.

More formally, we construct an \gls*{acr:ot} ambiguity set $\ball{\varepsilon}{\norm{\cdot}_2^2}{\widehat\P}$ centered at the empirical distribution $\widehat\P = \frac{1}{N}\sum_{i=1}^N\delta_{\sample{z}{i}}$, which is guaranteed to contain the true probability distribution $\P$ with high probability and study the resulting \gls*{acr:ot} ambiguity set around the estimation error $\hat{x} - x_0$.

\begin{remark}[i.i.d. case]
If the noise measurements $z_1,\ldots,z_m$ are i.i.d. according to an unknown noise distribution $\P\in\Pp{}{\reals}$, and we have access to $N$ samples from $\P$, then \cite[Theorem~2]{Fournier2015} can be used to construct an \gls{acr:ot} ambiguity set which contains $\P$ with high probability. 
Subsequently,~\cite[Lemma 3]{aolaritei2023capture} can be used to construct an \gls{acr:ot} ambiguity set for the probability distribution of $z$ (i.e., the product distribution $\P^{\otimes m}=\P\otimes\ldots\otimes\P$). %This OT ambiguity set will be centered at the product distribution $\empiricalP{N}^{\otimes m}$ with radius $m\varepsilon$.
%In particular, if $\P\in\ball{\varepsilon}{\norm{\cdot}_2^2}{\empiricalP{N}}$, then $\P^{\otimes m}\in\ball{m\varepsilon}{\norm{\cdot}_2^2}{{\empiricalP{N}}^{\otimes m}}$.
\end{remark}

Our results readily allow us to quantify the distributional uncertainty in the~\gls{acr:ols} estimator.

\begin{proposition}[Distributional uncertainty in \gls{acr:ols} estimator]
\label{prop:leastsquares}
Consider the linear inverse problem \eqref{eq:least:squares}, with $y \in \reals^m$, $x_0 \in \reals^n$ (deterministic or random), and $z \sim \P$. Moreover, let the (known) matrix $A$ be full column-rank. Finally, 
let $\P$ belong to the ambiguity set $\ball{\varepsilon}{\norm{\cdot}_2^2}{\widehat\P}$ constructed around the empirical distribution $\widehat\P = \frac{1}{N} \sum_{i=1}^N \delta_{\sample{z}{i}}$, for some $\varepsilon>0$. Then, the \gls{acr:ols} estimator $\hat{x}$ from \eqref{eq:least:squares:estimator} satisfies
\begin{align}
\label{eq:prop:leastsquares}
    \hat{x} - x_0 \sim \Q \in \ball{\varepsilon}{\norm{\cdot}_2^2 \circ A}{\pushforward{\pinv{A}} \widehat\P}.
\end{align}
\end{proposition}

\begin{remark}[i.i.d. case, continued]
In the i.i.d. setting of the previous remark, we therefore conclude that $\hat{x} - x_0 \sim \Q \in \ball{m\varepsilon}{\norm{\cdot}_2^2 \circ A}{\pushforward{\pinv{A}} \widehat\P^{\otimes m}}$.
\end{remark}

When the measurement noise has zero mean and variance $\sigma^2_z$, the covariance matrix of the \gls{acr:ols} estimator is equal to $\inv{(\T{A}A)}\sigma_z^2$ \cite{sullivan2015introduction}.
However, except for basic cases (e.g., Gaussian measurement noise), the covariance does not entirely capture the distributional uncertainty in the \gls{acr:ols} estimator. On the contrary, \cref{prop:leastsquares} enables us to fully capture the distributional uncertainty in the estimation error $\hat{x}  - x_0$.

\subsection{Nonlinear Dynamical Systems}\label{sec:app:dynamical:nonlinear}

Our last application concerns the propagation of the distributional uncertainty through a nonlinear dynamical system. Consider the deterministic dynamical system
\begin{align}
\label{eq:nonlinear:dyn:sys}
    x_{t+1} = f(x_t), \quad x_0 \sim \Q_0 \in \ball{\varepsilon}{c}{\P_0},
\end{align}
with $f:\reals^n \to \reals^n$. %The notation $x^+$ stands for either discrete-time, in which case \eqref{eq:nonlinear:dyn:sys} is a difference equation, i.e. $x_{t+1} = f(x_t)$, or continuous-time, in which case \eqref{eq:nonlinear:dyn:sys} is an ordinary differential equation, i.e. $\dot{x} = f(x)$. 
%We denote by $\Phi_t(x_0)$ the solution map of the dynamical system \eqref{eq:nonlinear:dyn:sys} at time $t$, parametrized by the initial condition $x_0$.
%Later in this section, we will also consider the case of stochastic dynamical systems.
We suppose that the initial condition $x_0$ is distributed according to some unknown probability distribution $\Q_0$, which belongs to an ambiguity set $\ball{\varepsilon}{c}{\P_0}$, where, as above, $\P_0$ could be an empirical distribution over finitely many samples. 
\cref{thm:nonlin:trans} allows us to propagate the distributional uncertainty encapsulated in the \gls{acr:ot} ambiguity set $\ball{\varepsilon}{c}{\P_0}$ and, thus, to characterize the distributional uncertainty in $x_t$.

\begin{corollary}
\label{cor:prop:nonlin:dyn}
Consider the dynamical system~\eqref{eq:nonlinear:dyn:sys}.
If $f$ is injective with left inverse $f^{-1}$, then
\begin{align*}
    x_t \sim \Q_t \in
    \ball{\varepsilon}{c \circ (f^{-t} \times f^{-t})}{\pushforward{f^{t}}\P},
\end{align*}
with $f^{t} = f \circ \ldots \circ f$ and $f^{-t} = f^{-1} \circ \ldots \circ f^{-1}$.
\end{corollary}

Unfortunately, the nonlinearity of~\eqref{eq:nonlinear:dyn:sys} prevents us from replicating our results in~\cref{prop:applications:propagation:linear} for the \gls{acr:lti} case (in particular, additive and multiplicative noise) without resorting to Lipschitz bounds.
For instance, in the case of additive process noise, there is no explicit expression for the state $x_t$ as a function of $x_0$ and the noise trajectory $\mathbf{w}_{[t-1]}$, which was crucial to prove 2) in~\cref{prop:applications:propagation:linear}. We leave the analysis of additive and multiplicative noise in the context of nonlinear dynamical systems to future research, and conclude this section with an outlook on some exciting \emph{open problems}.

\smallskip

\textbf{1) Distributionally robust control}. Our results can be directly employed in a variety of Distributionally Robust Optimal Control and MPC formulations, offering many advantages:
\begin{itemize}
    \item[(i)] \emph{Computation}. Propagating \gls{acr:ot} ambiguity sets from the noise to the state leads to much smaller optimization problems (as in \cref{prop:app:distributionally robust trajectory:reformulation}), whose dimension is independent of the control horizon. Moreover, we envision that a disturbance affine feedback policy leads to convex reformulations of DRC problems, where it is possible to optimize over both the feedforward and feedback terms.

    \smallskip
    
    \item[(ii)] \emph{Analysis}. The propagated \gls{acr:ot} ambiguity set capturing the distributional uncertainty in the state (recall \cref{prop:applications:propagation:linear}) unveils the role of the feedforward and feedback terms in the control input $u_t = K x_t + v_t$: $v_t$ controls the \emph{position}, while $K$ controls the \emph{shape} and \emph{size} of this \gls{acr:ot} ambiguity set (see also \cite[Section~IV-B]{aolaritei2023capture}).
\end{itemize}

\smallskip

\textbf{2) Distributionally robust filtering}. We envision that our results can play a major role in developing and studying novel distributionally robust filtering formulations, e.g., ensemble Kalman filter or particle filter. Specifically, we believe that the propagation of finitely many samples coupled with a robustification via \gls{acr:ot} ambiguity sets (which captures the underlying true distribution) allows us to rigorously quantify uncertainty in such filtering algorithms.

\smallskip

\textbf{3) Robustify control in probability spaces}. We envision that the distributionally robust results proposed in this paper can serve as the right baseline to robustify existing control approaches in the probability space (e.g., steering of distributions); see \cite{chen2021optimal,adu2022optimal,singh2020inference,terpin2023dynamic,taghvaei2021optimal,krishnan2020probabilistic,emerick2023continuum,elamvazhuthi2018optimal,balci2022exact,sivaramakrishnan2022distribution}, and the references therein.

\smallskip

\textbf{4) Convergence, invariance, and stability analysis}. Finally, we strongly believe that our results pave the way to studying the properties of dynamical systems affected by distributional uncertainty. Specifically, our exact propagation results allow to \emph{exactly embed} such systems in the probability space as a sequence of \emph{controlled} \gls{acr:ot} ambiguity sets. Then, convergence and invariance properties can be naturally formulated as \emph{convergence and invariance of \gls{acr:ot} ambiguity sets}. Moreover, we envision that these results can be employed in studying the stability of systems affected by distributional uncertainty.

\bibliographystyle{unsrt}
\bibliography{references}

\begin{thebibliography}{10}

\bibitem{Peyman2018}
Peyman Mohajerin~Esfahani and Daniel Kuhn.
\newblock {Data-driven distributionally robust optimization using the
  Wasserstein metric: performance guarantees and tractable reformulations}.
\newblock {\em Mathematical Programming}, 171(1):115--166, 2018.

\bibitem{sinha2017certifying}
Aman Sinha, Hongseok Namkoong, Riccardo Volpi, and John Duchi.
\newblock Certifying some distributional robustness with principled adversarial
  training.
\newblock {\em arXiv preprint arXiv:1710.10571}, 2017.

\bibitem{Blanchet2019}
Jose Blanchet and Karthyek Murthy.
\newblock Quantifying distributional model risk via optimal transport.
\newblock {\em Mathematics of Operations Research}, 44(2):565--600, 4 2019.

\bibitem{gao2022finite}
Rui Gao.
\newblock Finite-sample guarantees for {W}asserstein distributionally robust
  optimization: Breaking the curse of dimensionality.
\newblock {\em Operations Research}, 2022.

\bibitem{shafieezadeh2023new}
Soroosh Shafieezadeh-Abadeh, Liviu Aolaritei, Florian D{\"o}rfler, and Daniel
  Kuhn.
\newblock New perspectives on regularization and computation in optimal
  transport-based distributionally robust optimization.
\newblock {\em arXiv preprint arXiv:2303.03900}, 2023.

\bibitem{yang2020wasserstein}
Insoon Yang.
\newblock Wasserstein distributionally robust stochastic control: A data-driven
  approach.
\newblock {\em IEEE Transactions on Automatic Control}, 66(8):3863--3870, 2020.

\bibitem{aolaritei2023wasserstein}
Liviu Aolaritei, Marta Fochesato, John Lygeros, and Florian D{\"o}rfler.
\newblock Wasserstein tube mpc with exact uncertainty propagation.
\newblock {\em arXiv preprint arXiv:2304.12093}, 2023.

\bibitem{zhong2023efficient}
Zhengang Zhong, Ehecatl~Antonio Del Rio-Chanona, and Panagiotis Petsagkourakis.
\newblock An efficient data-driven distributionally robust {MPC} leveraging
  linear programming.
\newblock In {\em 2023 American Control Conference (ACC)}, pages 2022--2027.
  IEEE, 2023.

\bibitem{micheli2022data}
Francesco Micheli, Tyler Summers, and John Lygeros.
\newblock Data-driven distributionally robust mpc for systems with uncertain
  dynamics.
\newblock In {\em 2022 IEEE 61st Conference on Decision and Control (CDC)},
  pages 4788--4793. IEEE, 2022.

\bibitem{fochesato2022data}
Marta Fochesato and John Lygeros.
\newblock Data-driven distributionally robust bounds for stochastic model
  predictive control.
\newblock In {\em 2022 IEEE 61st Conference on Decision and Control (CDC)},
  pages 3611--3616. IEEE, 2022.

\bibitem{mark2021data}
Christoph Mark and Steven Liu.
\newblock Data-driven distributionally robust {MPC}: An indirect feedback
  approach.
\newblock {\em arXiv preprint arXiv:2109.09558}, 2021.

\bibitem{coulson2021distributionally}
Jeremy Coulson, John Lygeros, and Florian D{\"o}rfler.
\newblock Distributionally robust chance constrained data-enabled predictive
  control.
\newblock {\em IEEE Transactions on Automatic Control}, 67(7):3289--3304, 2021.

\bibitem{navsalkar2023data}
Atharva Navsalkar and Ashish~R Hota.
\newblock Data-driven risk-sensitive model predictive control for safe
  navigation in multi-robot systems.
\newblock In {\em 2023 IEEE International Conference on Robotics and Automation
  (ICRA)}, pages 1442--1448. IEEE, 2023.

\bibitem{kim2023distributional}
Kihyun Kim and Insoon Yang.
\newblock Distributional robustness in minimax linear quadratic control with
  wasserstein distance.
\newblock {\em SIAM Journal on Control and Optimization}, 61(2):458--483, 2023.

\bibitem{mcallister2023inherent}
Robert~D McAllister and James~B Rawlings.
\newblock On the inherent distributional robustness of stochastic and nominal
  model predictive control.
\newblock {\em IEEE Transactions on Automatic Control}, 2023.

\bibitem{zolanvari2023iterative}
Alireza Zolanvari and Ashish Cherukuri.
\newblock Iterative risk-constrained model predictive control: A data-driven
  distributionally robust approach.
\newblock {\em arXiv preprint arXiv:2308.11510}, 2023.

\bibitem{shafieezadeh2018wasserstein}
Soroosh Shafieezadeh~Abadeh, Viet~Anh Nguyen, Daniel Kuhn, and Peyman~M
  Mohajerin~Esfahani.
\newblock Wasserstein distributionally robust kalman filtering.
\newblock {\em Advances in Neural Information Processing Systems}, 31, 2018.

\bibitem{wang2022distributionally}
Shixiong Wang.
\newblock Distributionally robust state estimation for nonlinear systems.
\newblock {\em IEEE Transactions on Signal Processing}, 70:4408--4423, 2022.

\bibitem{han2023distributionally}
Bingyan Han.
\newblock Distributionally robust kalman filtering with volatility uncertainty.
\newblock {\em arXiv preprint arXiv:2302.05993}, 2023.

\bibitem{lotidis2023wasserstein}
Kyriakos Lotidis, Nicholas Bambos, Jose Blanchet, and Jiajin Li.
\newblock Wasserstein distributionally robust linear-quadratic estimation under
  martingale constraints.
\newblock In {\em International Conference on Artificial Intelligence and
  Statistics}, pages 8629--8644. PMLR, 2023.

\bibitem{brouillon2023regularization}
Jean-S{\'e}bastien Brouillon, Florian D{\"o}rfler, and Giancarlo
  Ferrari-Trecate.
\newblock Regularization for distributionally robust state estimation and
  prediction.
\newblock {\em IEEE Control Systems Letters}, 2023.

\bibitem{aolaritei2023capture}
Liviu Aolaritei, Nicolas Lanzetti, and Florian D{\"o}rfler.
\newblock Capture, propagate, and control distributional uncertainty.
\newblock {\em arXiv preprint arXiv:2304.02235}, 2023.

\bibitem{aolaritei2022uncertainty}
Liviu Aolaritei, Nicolas Lanzetti, Hongruyu Chen, and Florian D{\"o}rfler.
\newblock Uncertainty propagation via optimal transport ambiguity sets.
\newblock {\em arXiv preprint arXiv:2205.00343}, 2022.

\bibitem{boskos2023high}
Dimitris Boskos, Jorge Cort{\'e}s, and Sonia Mart{\'\i}nez.
\newblock High-confidence data-driven ambiguity sets for time-varying linear
  systems.
\newblock {\em IEEE Transactions on Automatic Control}, 2023.

\bibitem{boskos2020data}
Dimitris Boskos, Jorge Cort{\'e}s, and Sonia Mart{\'\i}nez.
\newblock Data-driven ambiguity sets with probabilistic guarantees for dynamic
  processes.
\newblock {\em IEEE Transactions on Automatic Control}, 66(7):2991--3006, 2020.

\bibitem{VanParys2016}
Bart~PG Van~Parys, Daniel Kuhn, Paul~J Goulart, and Manfred Morari.
\newblock Distributionally robust control of constrained stochastic systems.
\newblock {\em IEEE Transactions on Automatic Control}, 61(2):430--442, 2015.

\bibitem{coppens2021data}
Peter Coppens and Panagiotis Patrinos.
\newblock Data-driven distributionally robust {MPC} for constrained stochastic
  systems.
\newblock {\em IEEE Control Systems Letters}, 6:1274--1279, 2021.

\bibitem{schuurmans2023general}
Mathijs Schuurmans and Panagiotis Patrinos.
\newblock A general framework for learning-based distributionally robust {MPC}
  of markov jump systems.
\newblock {\em IEEE Transactions on Automatic Control}, 2023.

\bibitem{dixit2022distributionally}
Anushri Dixit, Mohamadreza Ahmadi, and Joel~W Burdick.
\newblock Distributionally robust model predictive control with total variation
  distance.
\newblock {\em IEEE Control Systems Letters}, 6:3325--3330, 2022.

\bibitem{romao2023distributionally}
Licio Romao, Ashish~R Hota, and Alessandro Abate.
\newblock Distributionally robust optimal and safe control of stochastic
  systems via kernel conditional mean embedding.
\newblock {\em arXiv preprint arXiv:2304.00644}, 2023.

\bibitem{Villani2009a}
Cédric Villani.
\newblock {\em Optimal Transport: Old and New}.
\newblock Springer-Verlag Berlin Heidelberg, 2009.

\bibitem{gibbs2002choosing}
Alison~L Gibbs and Francis~Edward Su.
\newblock On choosing and bounding probability metrics.
\newblock {\em International statistical review}, 70(3):419--435, 2002.

\bibitem{aolaritei2023hedging}
Liviu Aolaritei, Boubacar Bangoura, Nicolas Lanzetti, Saverio Bolognani, and
  Florian D{\"o}rfler.
\newblock Hedging against black swans in renewable energy markets via
  distributionally robust optimization.
\newblock {\em Work in Progress}, 2023.

\bibitem{bullo2020lectures}
Francesco Bullo.
\newblock {\em Lectures on network systems}, volume~1.
\newblock Kindle Direct Publishing Seattle, DC, USA, 2020.

\bibitem{sullivan2015introduction}
Timothy~John Sullivan.
\newblock {\em Introduction to uncertainty quantification}, volume~63.
\newblock Springer, 2015.

\bibitem{summers2021distributionally}
Tyler Summers and Maryam Kamgarpour.
\newblock Distributionally robust bootstrap optimization.
\newblock {\em arXiv preprint arXiv:2112.13932}, 2021.

\bibitem{Fournier2015}
Nicolas Fournier and Arnaud Guillin.
\newblock On the rate of convergence in {W}asserstein distance of the empirical
  measure.
\newblock {\em Probability Theory and Related Fields}, 162(3-4):707--738, 8
  2015.

\bibitem{chen2021optimal}
Yongxin Chen, Tryphon~T Georgiou, and Michele Pavon.
\newblock Optimal transport in systems and control.
\newblock {\em Annual Review of Control, Robotics, and Autonomous Systems},
  4:89--113, 2021.

\bibitem{adu2022optimal}
Daniel~Owusu Adu, Tamer Ba{\c{s}}ar, and Bahman Gharesifard.
\newblock Optimal transport for a class of linear quadratic differential games.
\newblock {\em IEEE Transactions on Automatic Control}, 67(11):6287--6294,
  2022.

\bibitem{singh2020inference}
Rahul Singh, Isabel Haasler, Qinsheng Zhang, Johan Karlsson, and Yongxin Chen.
\newblock Inference with aggregate data: An optimal transport approach.
\newblock {\em arXiv preprint arXiv:2003.13933}, 2020.

\bibitem{terpin2023dynamic}
Antonio Terpin, Nicolas Lanzetti, and Florian D{\"o}rfler.
\newblock Dynamic programming in probability spaces via optimal transport.
\newblock {\em arXiv preprint arXiv:2302.13550}, 2023.

\bibitem{taghvaei2021optimal}
Amirhossein Taghvaei and Prashant~G Mehta.
\newblock Optimal transportation methods in nonlinear filtering.
\newblock {\em IEEE Control Systems Magazine}, 41(4):34--49, 2021.

\bibitem{krishnan2020probabilistic}
Vishaal Krishnan and Sonia Mart{\'\i}nez.
\newblock A probabilistic framework for moving-horizon estimation: Stability
  and privacy guarantees.
\newblock {\em IEEE Transactions on Automatic Control}, 66(4):1817--1824, 2020.

\bibitem{emerick2023continuum}
Max Emerick and Bassam Bamieh.
\newblock Continuum swarm tracking control: A geometric perspective in
  {W}asserstein space.
\newblock {\em arXiv preprint arXiv:2303.15638}, 2023.

\bibitem{elamvazhuthi2018optimal}
Karthik Elamvazhuthi, Piyush Grover, and Spring Berman.
\newblock Optimal transport over deterministic discrete-time nonlinear systems
  using stochastic feedback laws.
\newblock {\em IEEE control systems letters}, 3(1):168--173, 2018.

\bibitem{balci2022exact}
Isin~M Balci and Efstathios Bakolas.
\newblock Exact {SDP} formulation for discrete-time covariance steering with
  {W}asserstein terminal cost.
\newblock {\em arXiv preprint arXiv:2205.10740}, 2022.

\bibitem{sivaramakrishnan2022distribution}
Vignesh Sivaramakrishnan, Joshua Pilipovsky, Meeko Oishi, and Panagiotis
  Tsiotras.
\newblock Distribution steering for discrete-time linear systems with general
  disturbances using characteristic functions.
\newblock In {\em 2022 American Control Conference (ACC)}, pages 4183--4190.
  IEEE, 2022.

\bibitem{Santambrogio2015}
Filippo Santambrogio.
\newblock {\em Optimal Transport for Applied Mathematicians}.
\newblock Birkh{\"{a}}user, Cham, 2015.

\end{thebibliography}

\appendix
\section{Proofs}
\subsection{Proofs for~\cref{sec:propagation}}
\begin{proof}[Proof of~\cref{lemma:stability set plans}]
We start with the proof of the inclusion $\pushforward{(f\times f)}\setPlans{\P}{\Q}
\subset \setPlans{\pushforward{f} \P}{\pushforward{f} \Q}$. We consider $\gamma \in \setPlans{\P}{\Q}$, and we want to prove that $\pushforward{(f\times f)} \gamma \in \setPlans{\pushforward{f}\P}{\pushforward{f} \Q}$, i.e., the marginals of $\pushforward{(f\times f)} \gamma$ are $\pushforward{f} \P$ and $\pushforward{f} \Q$. For any Borel and bounded test function $\phi:\mathcal Y \to \reals$, we have that
\begin{align*}
    \int_{\mathcal Y \times \mathcal Y} \!\phi(y_1) \d(\pushforward{(f\!\times \!f)}\gamma)(y_1,y_2)
    &=
    \int_{\mathcal X \times \mathcal X}\! \phi(f(x_1))\d\gamma(x_1,x_2) 
    \\
    =
    \int_{\mathcal X} \phi(f(x_1))\d\P(x_1) 
    &=
    \int_{\mathcal Y} \phi(y_1)\d(\pushforward{f}\P)(y_1),
\end{align*}
and therefore $\pushforward{{\proj{1}}} (\pushforward{(f\times f)}\gamma)=\pushforward{f} \P$, i.e., the first marginal of $\pushforward{(f\times f)} \gamma$ is $\pushforward{f} \P$. Analogously,
$\pushforward{{\proj{2}}}(\pushforward{(f\times f)}\gamma)=\pushforward{f} \Q$.
%This concludes the proof of the first inclusion.

We now prove $\setPlans{\pushforward{f} \P}{\pushforward{f} \Q} \subset \pushforward{(f\times f)} \setPlans{\P}{\Q}$. We consider $\gamma \in \setPlans{\pushforward{f} \P}{\pushforward{f} \Q}$ and seek $\tilde\gamma\in\setPlans{\P}{\Q}$ such that $\pushforward{(f\times f)} \tilde\gamma=\gamma$. For this, let $\gamma_{12}\coloneqq\pushforward{(\tensorProd{\Id_{\mathcal{X}}}{f})}\P\in\setPlans{\P}{\pushforward{f}\P}$, $\gamma_{23}\coloneqq\gamma\in\setPlans{\pushforward{f}\P}{\pushforward{f}\Q},$ and $
\gamma_{34}\coloneqq\pushforward{(\tensorProd{f}{\Id_{\mathcal X}})}\Q\in\setPlans{\pushforward{f}\Q}{\Q}$.
Since $\pushforward{\proj{2}}\gamma_{12}=\pushforward{\proj{1}}\gamma_{23}$ and $\pushforward{\proj{1}}\gamma_{34}=\pushforward{\proj{2}}\gamma_{23}$, the Gluing Lemma \cite[Lemma~5.5]{Santambrogio2015} ensures the existence of a measure $\bar\gamma\in\probSpace{\mathcal X\times \mathcal Y\times \mathcal Y\times \mathcal X}$ satisfying
$
    \pushforward{(\tensorProd{\proj{1}}{\proj{2}})}\bar\gamma=\gamma_{12},
    \pushforward{(\tensorProd{\proj{2}}{\proj{3}})}\bar\gamma=\gamma_{23},$ 
    and 
$
    \pushforward{(\tensorProd{\proj{3}}{\proj{4}})}\bar\gamma=\gamma_{34}.
$
In particular, $y_1=f(x_1)$ $\bar\gamma$-a.e., since 
\begin{align*}
    \int_{\mathcal X\times \mathcal Y}&\norm{y_1-f(x_1)}^2\d\bar\gamma(x_1,y_1,y_2,x_2)
    \\
    &=
    \int_{\mathcal X\times \mathcal Y}\norm{y_1-f(x_1)}^2\d\gamma_{12}(x_1,y_1)
    \\
    &=
    \int_{\mathcal X\times \mathcal Y}\norm{y_1-f(x_1)}^2\d(\pushforward{(\tensorProd{\Id_{\mathcal X}}{f})}\P)(x_1,y_1)
    \\
    &=
    \int_{\mathcal X}\norm{f(x_1)-f(x_1)}^2\d\P(x_1) =0.
\end{align*}
Analogously, $y_2=f(x_2)$ $\bar\gamma$-a.e..
Consider now $\tilde\gamma\coloneqq\pushforward{(\tensorProd{\proj{1}}{\proj{4}})}\bar\gamma$. We show that this is the probability distribution that we are looking for. In particular, $\tilde\gamma\in\setPlans{\P}{\Q}$ since, for any Borel and bounded test function $\phi:\mathcal X\to\reals$, we have that
\begin{align*}
    \int_{\mathcal X \times \mathcal X}\phi(x_1) &\d\tilde\gamma(x_1,x_2)
    =
    \int_{\mathcal X\times \mathcal Y\times \mathcal Y\times \mathcal X}\!\phi(x_1)\d\bar\gamma(x_1,y_1,y_2,x_2)
    \\
    &=
    \int_{\mathcal X \times \mathcal Y}\phi(x_1)\d\gamma_{12}(x_1,y_1)
    =
    \int_{\mathcal X}\phi(x_1)\d\P(x_1),
\end{align*}
which shows that $\pushforward{\proj{1}}\tilde\gamma=\P$. Analogously, $\pushforward{\proj{2}}\tilde\gamma=\Q$.
\begin{comment}
and analogously,
\begin{align*}
    \int_{\mathcal X \times \mathcal X}\phi(x_2) &\d\tilde\gamma(x_1,x_2)
    =
    \int_{\mathcal X\times \mathcal Y\times \mathcal Y\times \mathcal X}\phi(x_2)\d\bar\gamma(x_1,y_1,y_2,x_2)
    \\
    &=
    \int_{\mathcal X \times \mathcal Y}\phi(x_2)\d\bar\gamma_{34}(y_2,x_2)
    =
    \int_{\mathcal X}\phi(x_2)\d\Q(x_2),
\end{align*}
which shows that $\pushforward{\proj{1}}\tilde\gamma=\P$ and $\pushforward{\proj{2}}\tilde\gamma=\Q$, respectively.
\end{comment}
Finally, we show that $\pushforward{(f\times f)} \tilde\gamma=\gamma$. This follows from
\begin{align*}
    &\int_{\mathcal Y\times \mathcal Y}\phi(y_1,y_2)\d(\pushforward{(\tensorProd{f}{f})}\tilde\gamma)(y_1,y_2)
    \\
    &=
    \int_{\mathcal X\times \mathcal X}\phi(f(x_1),f(x_2))\d\tilde\gamma(x_1,x_2)
    \\
    &=
    \int_{\mathcal X\times \mathcal Y\times \mathcal Y\times \mathcal X}\phi(f(x_1),f(x_2))\d\bar\gamma(x_1,y_1,y_2,x_2)
    \\
    &=
    \int_{\mathcal X\times \mathcal Y\times \mathcal Y\times \mathcal X}\phi(y_1,y_2)\d\bar\gamma(x_1,y_1,y_2,x_2)
    \\
    &=
    \int_{\mathcal Y\times \mathcal Y}\phi(y_1,y_2)\d\gamma_{23}(y_1,y_2)
    =
    \int_{\mathcal Y\times \mathcal Y}\phi(y_1,y_2)\d\gamma(y_1,y_2),
\end{align*}
for any Borel and bounded test function $\phi:\mathcal Y \times \mathcal Y \to\reals$. Here, we used that $y_i=f(x_i)$ $\bar\gamma$-a.e.. This concludes the proof.
\end{proof}

\begin{proof}[Proof of~\cref{prop:Wc:pushforward}]
\cref{eq:W_c:arbitrary:f} follows from \cref{lemma:stability set plans}:
    \begin{align*}
        \transportCost{d}{\pushforward{f}\P}{\pushforward{f}\Q}\!
        &=\!\!
        \inf_{\gamma\in\setPlans{\pushforward{f}\P}{\pushforward{f}\Q}}\int_{\mathcal Y\times \mathcal Y}\!d(y_1,y_2)\d\gamma(y_1,y_2)
        \\
        &=\!\!
        \inf_{\gamma\in\pushforward{(\tensorProd{f}{f})}\setPlans{\P}{\Q}}\int_{\mathcal Y\times \mathcal Y}\!d(y_1,y_2)\d\gamma(y_1,y_2)
        \\
        &=\!\!
        \inf_{\tilde\gamma\in\setPlans{\P}{\Q}}\int_{\mathcal Y\times \mathcal Y}\!d(y_1,y_2)\d(\pushforward{(f\!\times\! f)}\tilde\gamma)(y_1,y_2)
        \\
        &=\!\!
        \inf_{\tilde\gamma\in\setPlans{\P}{\Q}}\int_{\mathcal X\times \mathcal X}d(f(x_1),f(x_2))\d\tilde \gamma(x_1,x_2)
        \\ 
        &=
        \transportCost{d\circ (\tensorProd{f}{f})}{\P}{\Q}.
    \end{align*}
The proofs of \eqref{eq:W_c:bijective:f} and \eqref{eq:W_c:injective:f} can be merged as follows. By definition $\inv{f}\circ f=\Id_{\mathcal X}$, and so $(\tensorProd{\inv{f}}{\inv{f}})\circ(\tensorProd{f}{f})=\Id_{\mathcal X\times \mathcal X}$.
    Thus,~\eqref{eq:W_c:arbitrary:f} with $d = c \circ (\inv{f} \times \inv{f})$ gives
    \begin{equation*}
    \begin{aligned}
        \transportCost{c}{\P}{\Q}
        &=
        \transportCost{c\circ(\tensorProd{\inv{f}}{\inv{f}})\circ(\tensorProd{f}{f})}{\P}{\Q}
        \\
        &=
        \transportCost{c\circ (\tensorProd{\inv{f}}{\inv{f}})}{\pushforward{f}\P}{\pushforward{f}\Q}. 
    \end{aligned}
    \end{equation*}
We now focus on \eqref{eq:W_c:surjective:f}, which can be proven as follows:
    \renewcommand{\qedsymbol}{}
    \begin{align*}
        &\transportCost{c\circ (\tensorProd{\inv{f}}{\inv{f}})}{\pushforward{f}\P}{\pushforward{f}\Q}
        \\
        &=
        \inf_{\gamma\in\setPlans{\pushforward{f}\P}{\pushforward{f}\Q}}\int_{\mathcal Y\times \mathcal Y}c(\inv{f}(y_1),\inv{f}(y_2))\d\gamma(y_1,y_2)
        \\
        &=
        \inf_{\gamma\in\pushforward{(\tensorProd{f}{f})}\setPlans{\P}{\Q}}\int_{\mathcal Y\times \mathcal Y}c(\inv{f}(y_1),\inv{f}(y_2))\d\gamma(y_1,y_2)
        \\
        &=
        \inf_{\tilde\gamma\in\setPlans{\P}{\Q}}\int_{\mathcal Y\times \mathcal Y}c(\inv{f}(y_1),\inv{f}(y_2))\d(\pushforward{(\tensorProd{f}{f})}\tilde\gamma)(y_1,y_2)
        \\
        &= \inf_{\tilde\gamma\in\setPlans{\P}{\Q}}\int_{\mathcal X\times \mathcal X}c(\inv{f}(f(x_1)),\inv{f}(f(x_2)))\d\tilde\gamma(x_1,x_2)
        \\
        &= \transportCost{c\circ (\tensorProd{\inv{f}}{\inv{f}}) \circ (\tensorProd{f}{f})}{\P}{\Q}.
        \hspace{4.2cm}\square\qedhere
    \end{align*}
\end{proof}

\begin{proof}[Proof of~\cref{thm:nonlin:trans}]
We start by proving \eqref{eq:B_epsilon:arbitrary:f}. Let $\mathbb Q \in \mathbb B_\varepsilon^{d \circ (f \times f)} (\P)$. To establish that $\pushforward{f} \mathbb Q \in \mathbb B_\varepsilon^{d}(\pushforward{f} \P)$, we use~\eqref{eq:W_c:arbitrary:f}:
\begin{multline}
%\begin{aligned}
\label{eq:nonlin:trans:proof:1}
    \mathbb Q \in \mathbb B_\varepsilon^{d \circ (f \times f)} (\P) 
    \iff 
    \transportCost{d\circ (\tensorProd{f}{f})}{\P}{\Q} \leq \varepsilon 
    \\
    \iff 
    \transportCost{d}{\pushforward{f}\P}{\pushforward{f}\Q}\leq \varepsilon 
    \iff 
    \pushforward{f} \mathbb Q \in \mathbb B_\varepsilon^{d}(\pushforward{f} \P).
%\end{aligned}
\end{multline}
Before proceeding with the proof, we highlight that the chain of equivalences \eqref{eq:nonlin:trans:proof:1} cannot be used to get equality of sets in \eqref{eq:B_epsilon:arbitrary:f}. The reason for this is that not all the probability distributions in the ambiguity set $\mathbb B_\varepsilon^{d}(\pushforward{f} \P)$ can be written as the pushforward $\pushforward{f}$ of a probability distribution over $\mathcal{X}$. This will become more clear for injective maps, in the proof of \eqref{eq:B_epsilon:injective:f}.

We will now prove the bijective case \eqref{eq:B_epsilon:bijective:f}. The inclusion $\pushforward{f} \ball{\varepsilon}{c}{\P} \subset \mathbb B_\varepsilon^{c \circ (\inv{f} \times \inv{f})}(\pushforward{f} \P)$ follows directly from \eqref{eq:B_epsilon:arbitrary:f}, by considering $d = c \circ (\inv{f} \times \inv{f})$, since $(\inv{f} \times \inv{f})\circ (f \times f) = \Id_{\mathcal X \times \mathcal X}$. We will now prove the converse inclusion, i.e., $\mathbb B_\varepsilon^{c \circ (\inv{f} \times \inv{f})}(\pushforward{f} \P) \subset \pushforward{f} \ball{\varepsilon}{c}{\P}$:
\begin{align*}
    &\ball{\varepsilon}{c \circ (\inv{f} \times \inv{f})}{\pushforward{f} \P}
    = 
    \pushforward{(f \circ \inv{f})} \ball{\varepsilon}{c \circ (\inv{f} \times \inv{f})}{\pushforward{f} \P}
    \\
    &\!=\! 
    \pushforward{f} \pushforward{\inv{f}\!} \ball{\varepsilon}{c \circ (\inv{f} \times \inv{f})}{\pushforward{f} \P}
    \!\subset\! 
    \pushforward{f} \ball{\varepsilon}{c}{\pushforward{\inv{f}\!} \pushforward{f} \P}
    \!=\!
    \pushforward{f} \ball{\varepsilon}{c}{\P},
\end{align*}
where the inclusion follows from \eqref{eq:B_epsilon:arbitrary:f} reversing $\mathcal X$ and $\mathcal Y$.

We now proceed to the injective case \eqref{eq:B_epsilon:injective:f}. Similarly to the bijective case, notice that the inclusion $\pushforward{f} \ball{\varepsilon}{c}{\P} \subset \mathbb B_\varepsilon^{c \circ (\inv{f} \times \inv{f})}(\pushforward{f} \P)$ follows directly from \eqref{eq:B_epsilon:arbitrary:f}, with $d = c \circ (\inv{f} \times \inv{f})$. Thus, we only have to prove that $\pushforward{f}{\mathbb B_\varepsilon^{c}(\P)} = \pushforward{(f \circ \inv{f})}\mathbb B_\varepsilon^{c\circ (\tensorProd{\inv{f}}{\inv{f}})} (\pushforward{f}{\P})$. Towards this end, we start by proving the inclusion $\pushforward{f}{\mathbb B_\varepsilon^{c}(\P)} \subset \pushforward{(f \circ \inv{f})}\mathbb B_\varepsilon^{c\circ (\tensorProd{\inv{f}}{\inv{f}})}(\pushforward{f}{\P})$. This follows immediately from $\pushforward{f} \ball{\varepsilon}{c}{\P} \subset \mathbb B_\varepsilon^{c \circ (\inv{f} \times \inv{f})}(\pushforward{f} \P)$ since
\begin{align*}
    \pushforward{f} \ball{\varepsilon}{c}{\P} 
    &= 
    \pushforward{(f \circ \inv{f} \circ f)}\ball{\varepsilon}{c}{\P} 
    = 
    \pushforward{(f \circ \inv{f})} \pushforward{f} \ball{\varepsilon}{c}{\P} 
    \\
    &\subset 
    \pushforward{(f \circ \inv{f})} \ball{\varepsilon}{c \circ (\inv{f} \times \inv{f})}{\pushforward{f} \P}.
\end{align*}
We now focus on the converse direction:
\begin{align*}
    \pushforward{(f \circ \inv{f})}
    \ball{\varepsilon}{c\circ (\tensorProd{\inv{f}\!}{\!\inv{f}})}{\pushforward{f}{\P}}
    &= 
    \pushforward{f} \pushforward{\inv{f}} \ball{\varepsilon}{c\circ (\tensorProd{\inv{f}\!}{\!\inv{f}})}{\pushforward{f}{\P}}
    \\
    \overset{\eqref{eq:B_epsilon:arbitrary:f}}{\subset}
    \pushforward{f} \ball{\varepsilon}{c}{\pushforward{\inv{f}}\pushforward{f}\P}
    &= \pushforward{f} \ball{\varepsilon}{c}{\P}.
\end{align*}
%where the inclusion follows again from \eqref{eq:B_epsilon:arbitrary:f}, with the roles of $\mathcal X$ and $\mathcal Y$ reversed.

We will now prove the surjective case \eqref{eq:B_epsilon:surjective:f}. We start with the inclusion $\mathbb B_\varepsilon^{c\circ (\tensorProd{\inv{f}}{\inv{f}})}(\pushforward{f}{\P}) \subset \pushforward{f}{\mathbb B_\varepsilon^{c}\left((\inv{f}\circ f)_\# \P\right)}$:
\begin{align*}
    &\mathbb B_\varepsilon^{c\circ (\tensorProd{\inv{f}}{\inv{f}})}(\pushforward{f}{\P})
    =
    \pushforward{(f \circ \inv{f})}\ball{\varepsilon}{c\circ (\tensorProd{\inv{f}}{\inv{f}})}{\pushforward{f}{\P}}
    \\
    &= 
    \pushforward{f} \pushforward{\inv{f}}\ball{\varepsilon}{c\circ (\tensorProd{\inv{f}}{\inv{f}})}{\pushforward{f}{\P}}
    \overset{\eqref{eq:B_epsilon:arbitrary:f}}{\subset}
    \pushforward{f} \ball{\varepsilon}{c}{\pushforward{\inv{f}}\pushforward{f}\P}
    \\
    &=
    \pushforward{f} \ball{\varepsilon}{c}{\pushforward{(\inv{f}\circ f)}\P}.
\end{align*}
%where the inclusion follows, as before, from \eqref{eq:B_epsilon:arbitrary:f}. 

Finally, we prove the equality $\pushforward{f}{\mathbb B_\varepsilon^{c\circ (\tensorProd{\inv{f}}{\inv{f}}) \circ (\tensorProd{f}{f})}(\P)} = \mathbb B_\varepsilon^{c\circ (\tensorProd{\inv{f}}{\inv{f}})}(\pushforward{f}{\P})$. The inclusion $\mathbb B_\varepsilon^{c\circ (\tensorProd{\inv{f}}{\inv{f}})}(\pushforward{f}{\P}) \supset \pushforward{f}{\mathbb B_\varepsilon^{c\circ (\tensorProd{\inv{f}}{\inv{f}}) \circ (\tensorProd{f}{f})}(\P)} $ follows directly from \eqref{eq:B_epsilon:arbitrary:f}, with $d = c \circ (\inv{f} \times \inv{f})$. 
We now focus on the converse inclusion, i.e., $\mathbb B_\varepsilon^{c\circ (\tensorProd{\inv{f}}{\inv{f}})}(\pushforward{f}{\P}) \subset \pushforward{f}{\mathbb B_\varepsilon^{c\circ (\tensorProd{\inv{f}}{\inv{f}}) \circ (\tensorProd{f}{f})}(\P)}$. Let $\mathbb Q_f \in \mathbb B_\varepsilon^{c\circ (\tensorProd{\inv{f}}{\inv{f}})}(\pushforward{f} \P)$. We need to prove that there exists $\mathbb Q \in \mathbb B_\varepsilon^{c\circ (\tensorProd{\inv{f}}{\inv{f}}) \circ (\tensorProd{f}{f})}(\P)$ such that $\mathbb Q_f = \pushforward{f} \mathbb Q$. We claim that one such $\Q$ is $\Q = \pushforward{\inv{f}} \Q_f$. Indeed, using \eqref{eq:W_c:surjective:f}, we have
    \begin{equation*}
    \begin{aligned}[b]
        &\hspace{-0.5cm}\transportCost{c\circ (\tensorProd{\inv{f}}{\inv{f}}) \circ (\tensorProd{f}{f})}{\P}{\pushforward{\inv{f}} \Q_f} 
        \\
        &= 
        \transportCost{c\circ (\tensorProd{\inv{f}}{\inv{f}})}{\pushforward{f}\P}{\pushforward{f} \pushforward{\inv{f}} \Q_f}
        \\
        &= 
        \transportCost{c\circ (\tensorProd{\inv{f}}{\inv{f}})}{\pushforward{f}\P}{\Q_f} 
        \leq 
        \varepsilon.
    \end{aligned}\qedhere 
    \end{equation*}
\end{proof}

\begin{comment}
\begin{proof}[Proof of~\cref{cor:lin:trans:specialcase}]
\begin{enumerate}
    \item Since a translation is bijective and translation-invariant, the result follows from \cref{thm:nonlin:trans}.
    
    \item Since scaling is a bijective map for any nonzero $\alpha$ and $c$ is homogeneous of degree $p$, \cref{thm:nonlin:trans} implies that  $
        \pushforward{f}{\ball{\varepsilon}{c}{\P}}=\ball{\varepsilon}{c/{|\alpha|^p}}{\pushforward{f}{\P}}$, for any $\alpha \neq 0$. Furthermore, the right-hand side is equivalent to stretching or compressing the radius of the \gls{acr:ot} ambiguity set while retaining the same transportation cost, i.e., $\ball{\abs{\alpha}^p\varepsilon}{c}{\pushforward{f}\P}$.
        
    \item The proof again follows \cref{thm:nonlin:trans} since $A$ is invertible. In particular, the transportation cost remains $c$ (instead of $c \circ A^{-1}$) since $c$ is rotation/reflection-invariant. That is,  $c(A (x_1-x_2)) = \psi(\|A (x_1-x_2)\|) = \psi(\|x_1-x_2\|) = c(x_1-x_2)$ due to the properties of $A$.
    
    \item Since $A$ is an orthogonal projection matrix, we have that $\pinv{A} = A$. Thus, the result follows from \cref{thm:lin:trans}.
    \qedhere 
\end{enumerate}
\end{proof}
\end{comment}

\begin{proof}[Proof of~\cref{thm:lin:trans}]
The linear structure, together with the translation invariance and orthomonotonicity of the transportation cost, allows us to prove the following relationship:
\begin{align}
\label{eq:lin:trans:proof:1}
    \pushforward{A}\ball{\varepsilon}{c}{\P}
    \subset
    \ball{\varepsilon}{c \circ \pinv{A}}{\pushforward{A}\P}
\end{align}
for an arbitrary matrix $A$. In the nonlinear case, we did not have such a relationship, since there is no pseudoinverse for an arbitrary nonlinear transformation $f$. Indeed, the equivalent of \eqref{eq:lin:trans:proof:1} for a nonlinear transformation is \eqref{eq:B_epsilon:arbitrary:f}.

We will now prove the inclusion~\eqref{eq:lin:trans:proof:1}. Let $\Q \in \ball{\varepsilon}{c}{\P}$. Then, $\pushforward{A} \Q \in \mathbb B_\varepsilon^{c \circ \pinv{A}}(\pushforward{A}\P)$ can be shown as follows:
\begin{align*}
    &\hspace{-.5cm}\inf_{\gamma\in\setPlans{\pushforward{A}\P}{\pushforward{A}\Q}}
    \int_{\reals^m\times \reals^m}c(\pinv{A} y_1-\pinv{A} y_2)\d\gamma(y_1,y_2)
    \\
    \overset{\text{\cref{lemma:stability set plans}}}&{=}
    \inf_{\gamma\in \pushforward{(A \times A)}\setPlans{\P}{\Q}}
    \int_{\reals^m\times \reals^m}c(\pinv{A} y_1-\pinv{A} y_2)\d\gamma(y_1,y_2)
    \\ &=
%    \inf_{\tilde{\gamma}\in \setPlans{\P}{\Q}}\int_{\reals^m\times \reals^m}c(\pinv{A} y_1,\pinv{A} y_2)\d\left((A \times A)_\# \tilde{\gamma}\right)(y_1,y_2)
%    \\ &=
    \inf_{\tilde{\gamma}\in \setPlans{\P}{\Q}}\int_{\reals^m\times \reals^m}c(\pinv{A}A x_1-\pinv{A}A x_2)\d\tilde{\gamma}(x_1,x_2)
    \\ &\leq
    \inf_{\tilde{\gamma}\in \setPlans{\P}{\Q}}\int_{\reals^m\times \reals^m}c(x_1- x_2)\d\tilde{\gamma}(x_1,x_2) \leq \varepsilon,
\end{align*}
where the inequality follows from orthomonotonicity of $c$ and $\pinv{A}A\in\reals^{n\times n}$ being the orthogonal projector onto $\kernel(A)^\perp=\range(\adjoint{A})$:
$
    c(x_1- x_2) = c(\pinv{A} A (x_1- x_2) + (\eye_{n\times n} - \pinv{A}A)(x_1- x_2)) 
    %\\
    \geq c(\pinv{A}A (x_1- x_2)).
$

With \eqref{eq:lin:trans:proof:1}, we can now prove the chain of relationships \eqref{eq:B_epsilon:arbitrary:A}. 
We start with the inclusion $\pushforward{A}{\ball{\varepsilon}{c}{\P}} \subset \pushforward{(A\pinv{A})} \mathbb B_\varepsilon^{c \circ \pinv{A}}(\pushforward{A} \P)$:
\begin{align*}
    \pushforward{A} \ball{\varepsilon}{c}{\P} &= \pushforward{(A \pinv{A} A)}\ball{\varepsilon}{c}{\P} = \pushforward{(A\pinv{A})} \pushforward{A} \ball{\varepsilon}{c}{\P} 
    \\
    \overset{\eqref{eq:lin:trans:proof:1}}&{\subset} \pushforward{(A\pinv{A})} \mathbb B_\varepsilon^{c \circ \pinv{A}}(\pushforward{A}\P),
\end{align*}
where the first equality follows from the definition of pseudoinverse ($A = A \pinv{A} A$).
We now prove the equality $\pushforward{(A\pinv{A})} \mathbb B_\varepsilon^{c \circ \pinv{A}}(\pushforward{A} \P) = \pushforward{A} \mathbb B_\varepsilon^{c}(\pushforward{(\pinv{A} A)} \P)$.
The inclusion $\pushforward{(A\pinv{A})} \mathbb B_\varepsilon^{c \circ \pinv{A}}(\pushforward{A} \P)\subset\pushforward{A} \ball{\varepsilon}{c}{\pushforward{(\pinv{A} A)} \P}$ follows from
\begin{multline*}
    \pushforward{(A\pinv{A})} \ball{\varepsilon}{c \circ \pinv{A}}{\pushforward{A} \P}
    = \pushforward{A} \pushforward{\pinv{A}} \mathbb B_\varepsilon^{c \circ \pinv{A}}(\pushforward{A} \P) 
    \\
    \overset{\eqref{eq:B_epsilon:arbitrary:f}}{\subset}
    \pushforward{A} \ball{\varepsilon}{c}{\pushforward{\pinv{A}} \pushforward{A} \P}
    = \pushforward{A} \ball{\varepsilon}{c}{\pushforward{(\pinv{A} A)} \P}.
\end{multline*}
%where the above inclusion follows from \eqref{eq:B_epsilon:arbitrary:f}.
The converse inclusion follows from 
\begin{align*}
    \pushforward{A} &\mathbb B_\varepsilon^{c}(\pushforward{(\pinv{A} A)} \P) 
    = 
    \pushforward{(A \pinv{A} A)} \mathbb B_\varepsilon^{c}(\pushforward{(\pinv{A} A)} \P) 
    \\
    &=
    \pushforward{(A\pinv{A})} \pushforward{A} \mathbb B_\varepsilon^{c}(\pushforward{(\pinv{A} A)} \P)  
    \\
    \overset{\eqref{eq:lin:trans:proof:1}}&{\subset}
    \pushforward{(A\pinv{A})} \mathbb B_\varepsilon^{c \circ \pinv{A}}(\pushforward{(A \pinv{A} A)} \P)
     =
    \pushforward{(A\pinv{A})} \mathbb B_\varepsilon^{c \circ \pinv{A}}(\pushforward{A} \P).
\end{align*}
%where the inclusion follows from \eqref{eq:lin:trans:proof:1}.
Finally, $\pushforward{A} \mathbb B_\varepsilon^{c}(\pushforward{(\pinv{A} A)} \P) \subset \mathbb B_\varepsilon^{c \circ \pinv{A}}(\pushforward{A}\P)$ follows from \eqref{eq:lin:trans:proof:1}:
\begin{align*}
    \pushforward{A} \mathbb B_\varepsilon^{c}(\pushforward{(\pinv{A} A)} \P) 
    \subset \mathbb B_\varepsilon^{c \circ \pinv{A}}(\pushforward{(A \pinv{A} A)} \P) 
    = \mathbb B_\varepsilon^{c \circ \pinv{A}}(\pushforward{A} \P).
\end{align*}
This concludes the proof for a general matrix $A$.

We now focus on the full row-rank case and show 
$
    \pushforward{A}\ball{\varepsilon}{c}{\P} 
    = 
    \ball{\varepsilon}{c \circ \pinv{A}}{\pushforward{A}\P},
$
with $\pinv{A}$ being the right inverse of $A$. In virtue of \eqref{eq:B_epsilon:arbitrary:A}, it suffices to prove that
$
    \ball{\varepsilon}{c \circ \pinv{A}}{\pushforward{A}\P}
    \subset
    \pushforward{A}\ball{\varepsilon}{c}{\P}.
$
Let $\Q\in\ball{\varepsilon}{c\circ\pinv{A}}{\pushforward{A}\P}$ and $\gamma\in\setPlans{\pushforward{A}\P}{\Q}$ be an optimal plan satisfying
$
    \int_{\reals^m\times\reals^m}c(\pinv{A}y_1-\pinv{A}y_2)\d\gamma(y_1,y_2)\leq\varepsilon.
$
Define the coupling $\bar\gamma=\pushforward{(\pinv{A}\times\pinv{A})}\gamma\in\setPlans{\pushforward{\pinv{A}A}\P}{\pushforward{\pinv{A}}\Q}$.
For clarity, we define by $x$, $x_i$ points in $\reals^n$, by $u$, $u_i$ points in $\mathcal U\coloneqq\range(\adjoint{A})$, and by $v$, $v_i$ points in $\mathcal V\coloneqq\kernel(A)$. By the Finite Rank Lemma, $\reals^n=\mathcal U\oplus \mathcal V$ and $\mathcal U^\perp= \mathcal V$, i.e., each $x$ can be uniquely decomposed in $u+v$ with $\innerProduct{u}{v}_X=0$ and $\pinv{A}A$ is the orthogonal projection on $\mathcal U$.
Thus, we see $\pinv{A}$ and $\pinv{A}A$ as maps $\reals^m \to \mathcal U$ and $\reals^n \to \mathcal U$, respectively.
By the Disintegration Theorem~\cite[\S 2.3]{Santambrogio2015}, there exists a $(\pushforward{(\pinv{A}A)}\P)$-a.e. uniquely determined family of probability distributions $\{\P_{u}\}_{u\in \mathcal U}$ on $\mathcal V$, such that 
$
    \d \P(u,v) = \d(\pushforward{(\pinv{A}A)}\P)(u) \otimes \d\P_{u}(v).
$
Similarly, there exists a $(\pushforward{(\pinv{A}A)}\P)$-a.e. uniquely determined family of probability distributions $\{(\bar\gamma_{u_1})\}_{u_1\in \mathcal U}$ on $\mathcal U$, such that
$
    \d\bar\gamma(u_1,u_2) = \d(\pushforward{(\pinv{A}A)}\P)(u_1) \otimes \d\bar\gamma_{u_1}(u_2).
$
Consider now the probability distribution $\bar\Q$ on $\reals^n$ defined by
\begin{equation*}
\begin{aligned}
    \d\bar\Q(u_1,v_1) \coloneqq \int_{\mathcal U} \d(\bar\gamma_{u_2} \otimes \P_{u_2})(u_1,v_1) \,\d(\pushforward{(\pinv{A}A)}\P)(u_2).
\end{aligned} 
\end{equation*}
We show that $\pushforward{A} \bar\Q = \Q$ and $\bar\Q \in \ball{\varepsilon}{c}{\P}$. Notice that this is enough to conclude the proof. For any Borel and bounded test function $\phi:\reals^m \to \reals$, we have
\begin{align*}
    \int_{\reals^m}
    &\phi(y)\d(\pushforward{A}\bar\Q)(y)  
    =
    \int_{\reals^n}
    \phi(A(u+v))\d\bar\Q(u,v)
    \\
    &=
    \int_{\reals^n}
    \phi(Au)\d\bar\Q(u,v)
    \\
    &=
    \int_{\reals^n \times \mathcal U}
    \phi(Au_1)\d(\bar\gamma_{u_2} \otimes \P_{u_2})(u_1,v_1) \,\d(\pushforward{(\pinv{A}A)}\P)(u_2)
    \\
    &=
    \int_{\mathcal U\times \mathcal U}
    \phi(Au_1)\d\bar\gamma_{u_2}(u_1) \,\d(\pushforward{(\pinv{A}A)}\P)(u_2)
    \\
    &=
    \int_{\mathcal U\times \mathcal U}
    \phi(Au_1)\d\bar\gamma(u_2,u_1)
    %=
    %\int_{\mathcal U\times \mathcal U}
    %\phi(Au_2)\d\bar\gamma(u_1,u_2)
    =
    \int_{\mathcal U}
    \phi(Au)\d(\pushforward{\pinv{A}}\Q)(u)
    \\
    &=
    \int_{\reals^m}
    \phi(A\pinv{A}y)\d\Q(y)
    =
    \int_{\reals^m}
    \phi(y)\d\Q(y),
\end{align*}
showing that $\pushforward{A} \bar\Q = \Q$. In the rest of the proof, we will show that $\bar\Q \in \ball{\varepsilon}{c}{\P}$. For this, we first define the coupling
\begin{align*}
    &\d\tilde \gamma(u_1,v_1,u_2,v_2)
    \coloneqq
    \pushforward{(\pi_{u_1} \times \pi_{v_1} \times \pi_{u_2} \times \pi_{v_2})}
    \\
    &\!\!\left(\d(\pushforward{(\pinv{A}A)}\P)(u_1) \!\otimes\! \left(\pushforward{(\Id \times \Id)}\d\P_{u_1}\right)(v_1,v_2) \!\otimes\! \d\bar\gamma_{u_1}(u_2)\right)\!.
\end{align*}
With the test function $\phi(u_1,v_1,u_2,v_2)=\norm{v_1-v_2}$, we see $v_1=v_2$ $\tilde\gamma$-a.e.. We now claim $\tilde \gamma\in\setPlans{\P}{\bar\Q}$. Indeed, for any Borel and bounded test function $\phi:\reals^n \to \reals$, we have
\begin{align*}
    \int_{\reals^n\times\reals^n}&
    \phi(u_1,v_1)\d\tilde\gamma(u_1,v_1,u_2,v_2)
    \\
    &=
    \int_{\reals^n}\phi(u_1,v_1) \d(\pushforward{(\pinv{A}A)}\P \otimes\P_{u_1})(u_1, v_1)
    \\
    &=
    \int_{\reals^n}\phi(u_1,v_1)\d\P(u_1,v_1),
\end{align*}
showing that the first marginal of $\tilde\gamma$ is $\P$, and
\begin{align*}  &\int_{\reals^n\times\reals^n}\phi(u_2,v_2)\d\tilde\gamma(u_1,v_1,u_2,v_2)
    \\
    &=
    \int_{\mathcal U\times\reals^n}\phi(u_2,v_2)\d(\bar\gamma_{u_1} \otimes \d\P_{u_1})(u_2,v_2)\d(\pushforward{(\pinv{A}A)}\P)(u_1)
    \\
    &=
    \int_{\reals^n}\phi(u_2,v_2)\int_{\mathcal U}\d(\bar\gamma_{u_1} \otimes \d\P_{u_1})(u_2,v_2)\d(\pushforward{(\pinv{A}A)}\P)(u_1)
    \\
    &=
    \int_{\reals^n}\phi(u_2,v_2)\d\bar\Q(u_2,v_2),
\end{align*}
showing that the second marginal of $\tilde\gamma$ is $\bar \Q$.
Finally, $\bar\Q \in \ball{\varepsilon}{c}{\P}$ follows from
\begin{align*}
    \transportCost{c}{\P}{\bar\Q}
    &\leq
    \int_{\reals^n\times\reals^n}c(u_1+v_1-(u_2+v_2))\d\tilde\gamma(u_1,v_1,u_2,v_2)
    \\
    &=
    \int_{\reals^n\times\reals^n}c(u_1-u_2)\d\tilde\gamma(u_1,v_1,u_2,v_2)
    \\
    &=
    \int_{\mathcal U\times \mathcal U}c(u_1-u_2)\d(\pushforward{(\pinv{A}A)}\P \otimes \bar\gamma_{u_1})(u_1,u_2)
    \\
    &=
    \int_{\mathcal U\times \mathcal U}c(u_1-u_2)\d(\pushforward{(\pinv{A}\times\pinv{A})}\gamma)(u_1,u_2)
    \\
    &=
    \int_{\reals^m\times\reals^m}c(\pinv{A}y_1-\pinv{A}y_2)\d\gamma(y_1,y_2)
    \leq 
    \varepsilon.
\end{align*}
This concludes the proof of the inclusion
$
    \ball{\varepsilon}{c \circ \pinv{A}}{\pushforward{A}\P}
    \subset
    \pushforward{A}\ball{\varepsilon}{c}{\P},
$
and, with it, the proof of~\cref{thm:lin:trans}.
\end{proof}

\subsection{Proofs for~\cref{sec:additive:multiplicative}}

\begin{proof}[Proof of~\cref{lemma:coupling:convolution}]
Let $\gamma \in \setPlans{\P_1}{\P_2}$. We want to show that $\gamma \ast (\pushforward{(\Id\times\Id)}\Q)$ is a coupling between $\P_1\ast\Q$ and $\P_2\ast\Q$. For any Borel and bounded test function $\phi:\reals^ n \to \reals$, we have
\begin{align*}
    &\int_{\reals^ n \times \reals^ n} \phi(x_1) \d (\gamma \ast \pushforward{(\Id\times\Id)}\Q)(x_1,x_2)
    \\
    &=
    \int_{\reals^ n \times \reals^ n} \phi(\pi_{1}(x_1,x_2)) \d (\gamma \ast \pushforward{(\Id\times\Id)}\Q)(x_1,x_2)
    \\
    \overset{\text{Def.~\ref{def:convolution}}}&{=}
    \int_{\reals^ n \times \reals^ n \times \reals^ n \times \reals^ n} \phi(\pi_{1}(y_1+z_1,y_2+z_2)) \\ &\hspace{3.3cm} \d (\gamma\otimes \pushforward{(\Id\times\Id)}\Q)(y_1,y_2,z_1,z_2)
    \\ &=
    \int_{\reals^ n \times \reals^ n} \phi(y_1+z_1) \d (\P_1\otimes \Q)(y_1,z_1)
    \\
    \overset{\text{Def.~\ref{def:convolution}}}&{=}
    \int_{\reals^ n} \phi(x_1) \d (\P_1\ast \Q)(x_1),
\end{align*}
showing that the first marginal of $\gamma \ast (\pushforward{(\Id\times\Id)}\Q)$ is $\P_1\ast\Q$.
%Notice that the second and fourth equalities follow from the definition of the convolution operation.
Analogously, the second marginal of $\gamma \ast (\pushforward{(\Id\times\Id)}\Q)$ is $\P_2\ast\Q$.
%of $\setPlans{\P_1}{\P_2}\ast (\pushforward{(\Id\times\Id)}\Q)  \subset \setPlans{\P_1\ast\Q}{\P_2 \ast \Q}$.
\end{proof}

\begin{proof}[Proof of~\cref{prop:W:convolution}]\renewcommand{\qedsymbol}{}
The result follows from \cref{lemma:coupling:convolution} and \cref{assump:transportation:cost:conv}(i):
\begin{align*}
     &\transportCost{c}{\P_1\ast \Q}{\P_2\ast  \Q} 
     \\&= 
     \underset{\gamma \in \setPlans{\P_1\ast \Q}{\P_2\ast \Q}}{\inf} \int_{\reals^ n\times \reals^ n} c(x_1-x_2) \d {\gamma (x_1,x_2)}
     \\ &\leq
     \underset{\gamma \in \setPlans{\P_1}{\P_2}\ast \pushforward{(\Id\times\Id)}\Q}{\inf} \int_{\reals^ n\times \reals^ n} c(x_1-x_2)  \d {\gamma (x_1,x_2)}
     \\ &= 
     \underset{\tilde\gamma \in \setPlans{\P_1}{\P_2}}{\inf}  \int_{\reals^ n\times \reals^ n} c(x_1-x_2)  \d {(\tilde\gamma \ast \pushforward{(\Id\times\Id)}\Q)(x_1,x_2)}
     %\\ &= 
     %\underset{\tilde\gamma \in \setPlans{\P_1}{\P_2}}{\inf}  \int_{\reals^ n\times \reals^ n \times \reals^ n\times \reals^ n} c(y_1+z_1-y_2-z_2)  \d {(\tilde\gamma \otimes \pushforward{(\Id\times\Id)}\Q)(y_1,y_2,z_1,z_2)}
     \\ &= 
     \underset{\tilde\gamma \in \setPlans{\P_1}{\P_2}}{\inf}  \int_{\reals^ n\times \reals^ n \times \reals^ n} c(y_1\!+z\!-\!y_2\!-\!z)  \d {(\tilde\gamma \otimes\Q)(y_1,y_2,z)}
     \\ &= 
     \underset{\tilde\gamma \in \setPlans{\P_1}{\P_2}}{\inf}  \int_{\reals^ n\times \reals^ n \times \reals^ n} c(y_1-y_2)  \d {(\tilde\gamma \otimes \Q)(y_1,y_2,z)}
     \\ &=
     \underset{\tilde\gamma \in \setPlans{\P_1}{\P_2}}{\inf}  \int_{\reals^ n\times \reals^ n} c(y_1-y_2)  \d {\tilde\gamma(y_1,y_2)}
     =
     \transportCost{c}{\P_1}{\P_2}.
     \hspace{0.2cm}\square\qedhere 
\end{align*}
\end{proof}

\begin{proof}[Proof of~\cref{thm:conv:trans}]
To start, \cref{assump:transportation:cost:conv}(ii) implies that $\transportCost{c}{\cdot}{\cdot}^{\frac{1}{p}}$ satifies triangle inequality. The proof follows analogously to the one of triangle inequality of the type-$p$ Wasserstein distance~\cite[Lemma 5.4]{Santambrogio2015}.
Then, let $\tilde{\P} \in \ball{\varepsilon_1}{c}{\P}$ and $\tilde{\Q} \in  \ball{\varepsilon_2}{c}{\Q}$. This implies that $\transportCost{c}{\tilde{\P}}{\P} \leqslant \varepsilon_1$ and $\transportCost{c}{\tilde{\Q}}{\Q} \leqslant \varepsilon_2$. Then, 
\begin{equation*}
\begin{split}
    \transportCost{c}{\tilde{\P}\ast\tilde{\Q}}{\P\ast \Q}^\frac{1}{p}  
    &\leq
    \transportCost{c}{\tilde{\P}\ast\tilde{\Q}}{\P\ast \tilde{\Q}}^\frac{1}{p} +   \transportCost{c}{\P\ast\tilde{\Q}}{\P\ast \Q}^\frac{1}{p}
    \\
    \overset{\text{Prop.~\ref{prop:W:convolution}}}&{\leq}
    \transportCost{c}{\tilde{\P}}{\P}^\frac{1}{p} +   \transportCost{c}{\tilde{\Q}}{ \Q}^\frac{1}{p}
    \leq 
    \varepsilon_1^\frac{1}{p} + \varepsilon_2^\frac{1}{p},
\end{split}
\end{equation*}
where the first inequality follows from the triangle inequality. Thus, $\transportCost{c}{\tilde{\P}\ast\tilde{\Q}}{\P\ast \Q}\leq (\varepsilon_1^{1/p} + \varepsilon_2^{1/p})^p$ and \eqref{thm:conv:trans:upperbound} follows.
%This concludes the proof.
\end{proof}

\begin{proof}[Proof of~\cref{lemma:coupling:Hadamard}]
Let $\gamma \in \setPlans{\P_1}{\P_2}$. We want to show that $\gamma \odot (\pushforward{(\Id\times\Id)}\Q)$ is a coupling between $\P_1\odot\Q$ and $\P_2\odot\Q$. For any Borel and bounded test function $\phi:\reals^n \to \reals$, we have
\begin{align*}
    &\int_{\reals^n \times \reals^n} \phi(x_1) \d (\gamma \odot \pushforward{(\Id\times\Id)}\Q)(x_1,x_2)
    \\&=
    \int_{\reals^n \times \reals^n} \phi(\pi_{1}(x_1,x_2)) \d (\gamma \odot \pushforward{(\Id\times\Id)}\Q)(x_1,x_2)
    \\
    \overset{\text{Def.~\ref{def:Hadamard}}}&{=}
    \int_{\reals^n \times \reals^n \times \reals^n \times \reals^n} \phi(\pi_{1}(y_1 \cdot z_1,y_2 \cdot z_2)) 
    \\ &\hspace{3cm}\d (\gamma\otimes \pushforward{(\Id\times\Id)}\Q)(y_1,y_2,z_1,z_2)
    \\ &=
    \int_{\reals^n \times \reals^n} \phi(y_1 \cdot z_1) \d (\P_1\otimes \Q)(y_1,z_1)
    \\
    \overset{\text{Def.~\ref{def:Hadamard}}}&{=}
    \int_{\reals^n} \phi(x_1) \d (\P_1\odot \Q)(x_1),
\end{align*}
showing that the first marginal of $\gamma \odot (\pushforward{(\Id\times\Id)}\Q)$ is $\P_1\odot\Q$.
%Notice that the second and fourth equalities follow from the definition of the Hadamard product. 
Analogously, the second marginal is $\P_2\odot\Q$.
\end{proof}

\begin{proof}[Proof~\cref{prop:W:Hadamard}]\renewcommand{\qedsymbol}{}
The result follows from \cref{lemma:coupling:Hadamard} and~\cref{assump:transportation:cost:Hadamard}(i):
\begin{align*}
     &\transportCost{c}{\P_1\odot \Q}{\P_2\odot  \Q} 
     \\&= 
     \underset{\gamma \in \setPlans{\P_1\odot \Q}{\P_2\odot \Q}}{\inf} \int_{\reals^n\times \reals^n} c(x_1-x_2) \d {\gamma (x_1,x_2)}
     \\ &\leq
     \underset{\gamma \in \setPlans{\P_1}{\P_2}\odot \pushforward{(\Id\times\Id)}\Q}{\inf} \int_{\reals^n\times \reals^n} c(x_1-x_2)  \d {\gamma (x_1,x_2)}
     \\ &= 
     \underset{\tilde\gamma \in \setPlans{\P_1}{\P_2}}{\inf}  \int_{\reals^n\times \reals^n} c(x_1-x_2)  \d {(\tilde\gamma \odot \pushforward{(\Id\times\Id)}\Q)(x_1,x_2)}
     \\ &= 
     \underset{\tilde\gamma \in \setPlans{\P_1}{\P_2}}{\inf}  \int_{\reals^n\times \reals^n \times \reals^n\times \reals^n} c(y_1\cdot z_1-y_2\cdot z_2) 
     \\&\hspace{3.5cm} \d {(\tilde\gamma \otimes \pushforward{(\Id\times\Id)}\Q)(y_1,y_2,z_1,z_2)}
     \\ &= 
     \underset{\tilde\gamma \in \setPlans{\P_1}{\P_2}}{\inf}  \int_{\reals^n\times \reals^n \times \reals^n} c(y_1\cdot z-y_2\cdot z)  \d {(\tilde\gamma \otimes\Q)(y_1,y_2,z)}
     \\ &\leq 
     \underset{\tilde\gamma \in \setPlans{\P_1}{\P_2}}{\inf}  \int_{\reals^n\times \reals^n \times \reals^n}  c(y_1-y_2)  c(z) \d {(\tilde\gamma \otimes \Q)(y_1,y_2,z)}
     \\ &= 
     \mathbb{E}_{\Q}[c(x)]\underset{\tilde\gamma \in \setPlans{\P_1}{\P_2}}{\inf}  \int_{\reals^n\times \reals^n}  c(y_1-y_2)  \d {\tilde\gamma(y_1,y_2)}.
     \hspace{1.1cm}\square\qedhere 
     %\\ &=
     %\transportCost{c}{\P_1}{\P_2}\, \mathbb{E}_{\Q}[c(z,0)]
     %\\ &=
     %\transportCost{c}{\P_1}{\P_2}\, \transportCost{c}{\Q}{\diracDelta{0}}.
\end{align*}
\end{proof}

\begin{proof}[Proof of~\cref{thm:Hadamard}]
As in the proof of~\cref{thm:conv:trans},~\cref{assump:transportation:cost:Hadamard}(ii) ensures triangle inequality for $\transportCost{c}{\cdot}{\cdot}^\frac{1}{p}$. Then, let $\tilde{\P} \in \ball{\varepsilon_1}{c}{\P}$ and $\tilde{\Q} \in  \ball{\varepsilon_2}{c}{\Q}$; i.e., $\transportCost{c}{\tilde{\P}}{\P} \leqslant \varepsilon_1$ and $\transportCost{c}{\tilde{\Q}}{\Q} \leqslant \varepsilon_2$. Then,
\begin{align*}
    &\hspace{-.5cm}\transportCost{c}{\tilde{\P}\odot\tilde{\Q}}{\P\odot \Q}^\frac{1}{p} 
    \\
    &\leq
    \transportCost{c}{\tilde{\P}\odot\tilde{\Q}}{\P\odot \tilde{\Q}}^\frac{1}{p} +   \transportCost{c}{\P\odot\tilde{\Q}}{\P\odot \Q}^{\frac{1}{p}}
    \\
    \overset{\text{Prop.~\ref{prop:W:Hadamard}}}&{\leq}
    \transportCost{c}{\tilde{\P}}{\P}^\frac{1}{p}\mathbb{E}_{\tilde\Q}[c(x)]^\frac{1}{p}
    +
    \transportCost{c}{\tilde{\Q}}{ \Q}^\frac{1}{p} \mathbb{E}_{\P}[c(x)]^\frac{1}{p}
    \\
    &\leq 
    \varepsilon_1^\frac{1}{p}\transportCost{c}{\tilde\Q}{\diracDelta{0}}^\frac{1}{p}
    +
    \varepsilon_2^\frac{1}{p}\mathbb{E}_{\P}[c(x)]^\frac{1}{p}
    \\& \leq
    \varepsilon_1^\frac{1}{p}  \left(\transportCost{c}{\tilde{\Q}}{\Q}^\frac{1}{p}+ \transportCost{c}{\Q}{\delta_0}^\frac{1}{p} \right)+   \varepsilon_2^\frac{1}{p} \mathbb{E}_{\P}[c(x)]^\frac{1}{p}
    \\
    &\leq
    \varepsilon_1^\frac{1}{p}  \left(\varepsilon_2^\frac{1}{p}+ \mathbb{E}_{\Q}[c(x)]^\frac{1}{p} \right)+  \varepsilon_2^\frac{1}{p}  \mathbb{E}_{\P}[c(x)]^\frac{1}{p}
    \\
    &=
    \varepsilon_1^\frac{1}{p} \varepsilon_2^\frac{1}{p} +\varepsilon_1^\frac{1}{p}  \mathbb{E}_{\Q}[c(x)]^\frac{1}{p} + \varepsilon_2^\frac{1}{p} \mathbb{E}_{\P}[c(x)]^\frac{1}{p},
\end{align*}
where the first and the fourth inequality follow from the triangle inequality of the \gls{acr:ot} discrepancy.
This establishes~\eqref{eq:thm:hadamard}. %and concludes the proof. 
\end{proof}

\subsection{Proofs for~\cref{sec:applications}}

\begin{proof}[Proof of~\cref{prop:applications:propagation:linear}]
We prove the three cases separately.

1) Since $x_t=A^t x_0+\mathbf{B}_{t-1} \mathbf{u}_{[t-1]}$ and the transportation cost satisfies~\cref{assump:transportation:cost}, the result follows from~\cref{thm:lin:trans}.

2) Since $x_{t} = A^t x_0 + \mathbf{B}_{t-1} \mathbf{u}_{[t-1]} + \mathbf{D}_{t-1} \mathbf{w}_{[t-1]}$ and the transportation cost satisfies~\cref{assump:transportation:cost}, the result follows from~\cref{thm:lin:trans}.

3)  We prove the statement by induction. The base case $t=0$ is trivial. Suppose then the statement holds for $t\in\naturals$. Then, since the element-wise product induces the Hadamard product of distributions and the sum induces the convolution of distributions, we have
$
    \Q_{t+1}=(\Q^{(1)}\odot \pushforward{A}\Q_{t})\ast(\Q^{(2)}\odot\diracDelta{Bu_{t}}).
$
Since $\Q_{t}\in\ball{\rho_t}{\norm{\cdot}_2^2}{\P_t}$, we have that $\Q_{t+1}$ belongs to
\begin{align*}
    &\left(\ball{\varepsilon_1}{\norm{\cdot}_2^2}{\P^{(1)}}\odot\pushforward{A}\ball{\rho_t}{\norm{\cdot}_2^2}{\P_t}\right)
    \ast\left(\ball{\varepsilon_2}{\norm{\cdot}_2^2}{\P^{(2)}}\odot\diracDelta{Bu_{t}}\right)
    \\
    &\hspace{-0.1cm}\subset\!\!
    \left(\ball{\varepsilon_1}{\norm{\cdot}_2^2}{\P^{(1)}}\odot\ball{\rho_t}{\norm{\cdot}_2^2\circ\pinv{A}}{\pushforward{A}\P_t}\right)\!
    \ast\!\left(\ball{\varepsilon_2}{\norm{\cdot}_2^2}{\P^{(2)}}\odot\diracDelta{Bu_{t}}\right)
    \\
    &\hspace{-0.1cm}\subset\!\!
    \left(\ball{\varepsilon_1}{\norm{\cdot}_2^2}{\P^{(1)}}\odot\ball{\rho_t\sigma_\mathrm{\max}(A)^2}{\norm{\cdot}_2^2}{\pushforward{A}\P_t}\right)\!\ast\!\left(\ball{\varepsilon_2}{\norm{\cdot}_2^2}{\P^{(2)}}\odot\diracDelta{Bu_{t}}\right)
    \\
    &\hspace{-0.1cm}\subset \!\!
    \begin{aligned}[t]
    &\ball{(\!\sqrt{\varepsilon_1 \rho_t}\sigma_\mathrm{\max}(\!A)+\sqrt{\rho_tM_{\P^{(1)}}}\sigma_\mathrm{\max}(\!A)+\sqrt{\varepsilon_1M_{\pushforward{A}\P_t}})^2}{\norm{\cdot}_2^2}{\P^{(1)}\!\odot\!
    \pushforward{A}\P_t}
    \\
    &\ast\ball{\varepsilon_2\norm{Bu_t}^2}{\norm{\cdot}_2^2}{\P^{(2)}\odot\diracDelta{Bu_{t}}} \subset
    \ball{\rho_{t+1}}{\norm{\cdot}_2^2}{\P_{t+1}}.
    \end{aligned}
\end{align*}
where the first inclusion follows~\cref{thm:lin:trans}, the second from $\ball{\varepsilon}{\norm{\cdot}_2^2\circ\pinv{A}}{\pushforward{A}\P_t}\subset\ball{\varepsilon\sigma_\mathrm{\max}(A)}{\norm{\cdot}_2^2}{\pushforward{A}\P_t}$ (since $\norm{\cdot}_2^2\circ\pinv{A}\geq \sigma_\mathrm{min}(\pinv{A})^2\norm{\cdot}_2^2=\sigma_\mathrm{max}(A)^{-2}\norm{\cdot}_2^2$),
the third from \cref{thm:Hadamard} and \cref{cor:hadamard:known:noise} (since $\norm{\cdot}_2^2$ satisfies \cref{assump:transportation:cost:Hadamard} with $p=2$), and the last one from~\cref{thm:conv:trans} (since $\norm{\cdot}_2^2$ satisfies \cref{assump:transportation:cost:conv} with $p=2$). %This concludes the induction and proof.  
\end{proof}

\begin{proof}[Proof of~\cref{prop:app:consensus}]
We first claim that $\Q_t \to \Q_\infty\coloneqq\pushforward{h} \Q_0$, with $h$ as in the statement. If this is true, $\Q_\infty \in \pushforward{h} \ball{\varepsilon}{c}{\P_0}$, and the result follows from \cref{thm:lin:trans}. Since the transformation $x \mapsto \T{w} x$ is surjective and $\norm{\cdot}^2_2$ is translation-invariant and orthomonotone,~\eqref{eq:B_epsilon:surjective:A:1} implies that
\begin{align*}
    \pushforward{\T{w}} \ball{\varepsilon}{\norm{\cdot}_2^2}{\P_0} = \ball{\varepsilon}{\norm{\cdot}_2^2 \circ \pinv{(\T w)}}{\pushforward{\T{w}} \P_0},
\end{align*}
with $\pinv{(\T w)} = w(\T{w} w)^{-1}=w/\norm{w}_2^2$ being the right inverse of $\T w$. Notice that $\ball{\varepsilon}{\norm{\cdot}_2^2 \circ w/\norm{w}_2^2}{\pushforward{\T w} \P_0}$ is an \gls{acr:ot} ambiguity set of probability distributions over $\reals$. Then, $\pushforward{h} \Q_0 \in \pushforward{\ones_n}\ball{\varepsilon}{\norm{\cdot}_2^2 \circ w/\norm{w}_2^2}{\pushforward{\T{w}} \P_0}$.
We now prove $\Q_t \to \pushforward{h} \Q_0$. A standard modal decomposition of $A$ gives $A^t x \to h(x) = (\T{w} x) v$ for all $x\in\reals^n$. For any continuous and bounded test function $\phi: \reals^n \to \reals$, dominated convergence gives
$
    \int_{\reals^n} \phi(x) \d\Q_t(x) 
    = 
    \int_{\reals^n} \phi(A^tx) \d\Q_0(x)
    \to
    \int_{\reals^n} \phi(h(x)) \d\Q_0(x)
    =
    \int_{\reals^n} \phi(x) \d(\pushforward{h}\Q_0)(x).
$
Then, \eqref{eq:app:consensus:one} follows from
$(\norm{\cdot}_2^2\circ \frac{w}{\norm{w}^2})(x-y)
=\|\frac{w}{\norm{w}^2}x-\frac{w}{\norm{w}^2}y\|_2^2
=\frac{\abs{x-y}^2}{\norm{w}^2_2}$ and 
the definition of \gls{acr:ot} ambiguity set. If additionally $w=\frac{1}{n}\ones_n$, $\norm{w}_2^2=\frac{1}{n}$, which establishes~\eqref{eq:app:consensus:two}. %This concludes the proof. 
\end{proof}

\begin{proof}[Proof of~\cref{prop:leastsquares}]
By~\eqref{eq:least:squares:estimator}, $\hat{x} = \pinv{A} y = \pinv{A} (A x_0 + z) = x_0 + \pinv{A} z$ or, equivalently, $\hat{x} - x_0 = \pinv{A} z$, with $\pinv{A} = \inv{(\T A A)}\T A$ being the left inverse of $A$. Therefore, $\hat{x} - x_0$ is distributed as $\pinv{A} z$, i.e., $\Q = \pushforward{\pinv{A}}\P$. Since $A$ has full column-rank,  $\pinv{A}$ has full row-rank. Therefore, the equality \eqref{eq:B_epsilon:surjective:A:1} from~\cref{thm:lin:trans} yields
\begin{align*}
    \Q
    =
    \pushforward{\pinv{A}} \P 
    \in \pushforward{\pinv{A}} \ball{\varepsilon}{\norm{\cdot}_2^2}{\widehat\P}
    =
    \ball{\varepsilon}{\norm{\cdot}_2^2 \circ A}{\pushforward{\pinv{A}} \widehat\P},
\end{align*}
where we used that $A$ is the right inverse of $\pinv{A}$. %This concludes the proof.
\end{proof}

%\begin{proof}[Proof of~\cref{cor:prop:nonlin:dyn}]
%    The result follows from~\cref{thm:nonlin:trans}. 
%\end{proof}

\end{document}